\documentclass[journal,twoside,web]{ieeecolor}
\usepackage{generic}
\usepackage{cite}
\usepackage{amsmath,amssymb,amsfonts}
\usepackage{graphicx}
\usepackage{algorithm}
\usepackage{algpseudocode}
\usepackage{hyperref}
\hypersetup{hidelinks}
\usepackage{textcomp}
\def\BibTeX{{\rm B\kern-.05em{\sc i\kern-.025em b}\kern-.08em
    T\kern-.1667em\lower.7ex\hbox{E}\kern-.125emX}}
\allowdisplaybreaks

\usepackage{nomencl}
\makenomenclature

\newtheorem{theorem}{Theorem}
\newtheorem{lemma}{Lemma}
\newtheorem{remark}{Remark}

\newtheorem{assumption}{Assumption}

\newcommand{\li}{}
\newcommand{\ml}{}
\newcommand{\mml}{}
\newcommand{\icl}{}
\newcommand{\mll}{}

\usepackage{subcaption}
\author{Mengmou~Li, \IEEEmembership{Member, IEEE}, Ioannis~Lestas, \IEEEmembership{Member, IEEE}, and Masaaki~Nagahara,  \IEEEmembership{Senior Member, IEEE}
\thanks{This work is supported by JSPS KAKENHI under Grant Numbers 26K17401 and 24K21314. This work is also supported by Japan Science and Technology Agency (JST) as part of Adopting Sustainable Partnerships for Innovative Research Ecosystem (ASPIRE), Grant Number JPMJAP2402.
}
\thanks{M.~Li and M.~Nagahara are with the Graduate School of Advanced Science and Engineering, Hiroshima University, Higashi-Hiroshima City, Japan (e-mail: mmli.research@gmail.com; nagahara@ieee.org).}
\thanks{I.~Lestas is with the Department of Engineering, University of Cambridge, UK (icl20@cam.ac.uk).}
}
\title{First-Order Projected Algorithms With the Same Linear Convergence Rate Bounds as Their Unconstrained Counterparts}
\begin{document}

\maketitle
\begin{abstract}
\mll{In this paper, we propose a systematic approach for extending first-order optimization algorithms, originally designed for unconstrained strongly convex problems, to handle closed convex set constraints.}
\mll{We show that the resulting projected algorithms} retain the same linear convergence rate bounds, \mll{provided that the underlying unconstrained optimization algorithms admit a quadratic Lyapunov function obtained from integral quadratic constraint (IQC) analysis.}
%
The projected algorithms are constructed by applying a projection in the norm induced by the Lyapunov matrix, ensuring both constraint satisfaction and optimality at the fixed point. Furthermore, under a linear transformation associated with this matrix, the projection becomes non-expansive in the Euclidean norm, 
\li{thereby preserving the convergence rate bounds under the composition of the linearly convergent algorithmic operator and the projection.}
Our results indicate that, when analyzing worst-case convergence rates or \icl{when} synthesizing 
\icl{first-order optimization algorithms with potentially \mml{higher-order} dynamics,} it suffices to focus solely on the unconstrained dynamics, \mll{since the same parameters or stepsizes can be employed without retuning.}

\end{abstract}
\begin{IEEEkeywords}
Constrained optimization, projection operator, saturation, geometric convergence rate, integral quadratic constraint, quadratic Lyapunov function
\end{IEEEkeywords}
%
\IEEEpeerreviewmaketitle

\section{Introduction}
\IEEEPARstart{T}{he} design and analysis of various optimization algorithms have attracted significant attention in recent years. In particular, obtaining tight exponential convergence rate bounds for various optimization algorithms based on frequency-domain analysis has emerged as an active area of research \cite{lessard2016analysis,michalowsky2021robust,scherer2021convex,zhang2022zames,su2025exponential,zhang2024frequency}.
Frequency-domain analysis methods can be unified within the framework of integral quadratic constraints (IQCs)\cite{megretski1997system}. 
Considerable progress has been made in investigating the necessity of IQC analysis \cite{jonsson2001lecture,khong2020converse} and the role of the well-known O'Shea--Zames--Falb (OZF) multipliers within it \cite{carrasco2016zames,su2023necessity,gyotoku2024dual}. The latter also includes demonstrating the phase containment of other classes of multipliers within the OZF class, particularly the Popov multipliers that are unbounded at infinity
\cite{carrasco2013equivalence,carrasco2014multipliers,li2024generalization}.
The IQC framework is powerful in that less conservative convergence rate \mml{bounds} may be obtained in both continuous-time \cite{hu2016exponential} and discrete-time schemes \cite{boczar2015exponential}. It also gives insights to disprove the global convergence of the heavy-ball method \cite{lessard2016analysis}.
 \li{Through the well-established Kalman--Yakubovich--Popov (KYP) lemma \cite{rantzer1996kalman}, IQCs can also be formulated in the time domain and are closely connected to quadratic Lyapunov or storage functions when rational and bounded multipliers are used.}
The time-domain IQC analysis can be modified in the context of dissipativity to analyze sublinear rates under varying stepsizes \cite{hu2017dissipativity}.
Recently, efforts have also been made to synthesize accelerated algorithms within the IQC framework \cite{scherer2021convex}, \cite{scherer2023optimization}. A variant of Nesterov's algorithm, termed the triple momentum \li{method (TMM)}, has been proposed in \cite{van2017fastest} with improved convergence rates.
\mll{Meanwhile, there exists a parallel line of research to IQC-based analysis in the optimization community, termed performance estimation problems (PEPs), which formulates interpolation conditions for the convex gradient mappings \cite{taylor2017smooth,taylor2017exact}.}



In addition to unconstrained algorithms, the study of IQC-based analysis for the convergence rates of constrained optimization algorithms also receives much research attention \cite{dhingra2018proximal,hassan2021proximal,li2024exponential}.
The exponential stability of the continuous-time proximal augmented Lagrangian method has been investigated in \cite{dhingra2018proximal}, and similar stability results for the proximal gradient and Douglas-Rachford splitting flows are presented in \cite{hassan2021proximal}. 
The convergence rates of projected versions of Nesterov's accelerated method, the alternating direction method of multipliers (ADMM), and the mirror descent method have been studied in \cite{lessard2016analysis, lessard2022analysis, li2025convergence}, respectively. However, these works focus solely on existing projected algorithms and rely on numerical methods when frequency-domain inequalities (FDIs) in the IQC framework cannot be solved analytically, leaving them disconnected from their unconstrained counterparts. In the literature, analytical guarantees for projected algorithms that achieve the same convergence rates as their unconstrained counterparts have only been established for a few cases, such as projected gradient descent and its accelerated methods \cite{bertsekas1999nonlinear,beck2009fast,nesterov2013gradient}.
It is conjectured in \cite{lessard2016analysis} that any algorithm of the Lur'e form (\eqref{eq: LTI system}, \eqref{eq: gradient input to u_k} in this work) that converges has a proximal variant that converges at precisely the same rate. In this paper, we provide \li{a partial affirmative} answer to this conjecture \li{for a class of algorithms satisfying the assumptions stated later}.

\mll{The proposed framework in this work builds upon classical tools in control, such as Lyapunov functions \li{and} IQCs.}
\mll{We propose a systematic approach for extending first-order optimization algorithms, originally designed for unconstrained strongly convex problems, to handle closed convex set constraints.}
\mll{We show that the resulting projected algorithms} retain the same linear convergence rate bounds, \mll{provided that the underlying unconstrained optimization algorithms admit a quadratic Lyapunov function, \li{which can be verified by suitable IQC conditions}.}
%
The projected algorithms are constructed by applying a projection in the norm induced by the Lyapunov matrix, ensuring both constraint satisfaction and optimality at the fixed point. Furthermore, under a linear transformation associated with this matrix, the projection becomes non-expansive in the Euclidean norm, 
\li{allowing the preservation of the convergence rate bounds under the composition of the linearly convergent algorithmic operator and the projection.}
Our results indicate that, when analyzing worst-case convergence rates or synthesizing 
first-order \icl{optimization algorithms with \mml{higher-order} dynamics}, it suffices to focus solely on the unconstrained dynamics.
The contributions are listed below:
\begin{enumerate}
    \item \mll{We provide \li{a partial affirmative} answer to the conjecture by \cite{lessard2016analysis}, showing that for strongly convex optimization problems, \li{a broad class of} Lur’e-type unconstrained first-order algorithms admit a \li{projected} variant with the same rate bound, provided that the rate bound is established via a quadratic Lyapunov function derived from the IQC analysis.
    This result establishes a theoretical bridge between unconstrained and constrained optimization dynamics.}
    \item \mll{
    In contrast to existing studies that analyze the convergence rate bounds of specific projected algorithms \cite{lessard2016analysis, lessard2022analysis, li2025convergence}, our framework applies to a broad class of Lur’e-type algorithms. Specifically, it generates weighted projection steps in the norm induced by the Lyapunov matrix, which guarantees both constraint satisfaction and optimality at the fixed point. Under a corresponding linear transformation, the projection becomes non-expansive in the Euclidean norm, 
    \li{thereby preserving the convergence rate bounds under the composition of the linearly convergent algorithmic operator and the projection.}
    }
\end{enumerate}
\mll{It is worth noting that the same parameters or stepsizes as in the unconstrained case are employed, and no retuning is needed when extending algorithms to constrained problems. This property simplifies implementation while preserving the convergence rate \li{bounds}.}
\mll{The question of whether a quadratic Lyapunov function necessarily exists for establishing specific convergence rates is closely related to the necessity of IQC and OZF multipliers \cite{carrasco2016zames, su2023necessity}, which lies beyond the scope of this work.}
The analysis presented in this work is elementary and, in our view, can be easily extended to other algorithms beyond Lur’e systems, provided the existence of certain Lyapunov functions.
This perspective may also be relevant to anti-windup analysis \cite{galeani2009tutorial}, as saturation can be viewed as a special case of projection.
Finally, we restrict our attention to theoretical convergence rate bounds and do not address computational considerations such as the practical feasibility of executing the projection.

The rest of the paper is organized as follows.
In Section~\ref{Preliminaries}, we review knowledge of convex analysis, Lyapunov theory for exponential stability, linear convergence, and the IQC framework.
In Section~\ref{Main results}, projected algorithms are constructed and shown to achieve the same convergence rate \mml{bounds} as their unconstrained counterparts.
In Section~\ref{Examples}, we present the projected gradient descent and propose the projected triple momentum \li{method} as illustrative examples, and provide a numerical demonstration. Finally, the paper is concluded in Section~\ref{Conclusion}.

\section{Preliminaries}\label{Preliminaries}
$\mathbb{N}$, $\mathbb{R}$, $\mathbb{C}$, $\mathbb{R}^{n}$, $\mathbb{R}^{n \times m}$ denote the sets of natural numbers, real numbers, complex numbers, $n$-dimensional real vectors, and $n\times m$ real matrices, respectively.
The $n \times n$ identity matrix, $n \times n$, and $ n \times m$ zero matrices are denoted by $I_{n}$, $0_{n}$, and $0_{n \times m}$, respectively. Subscripts are omitted when they can be inferred from the context.
The block diagonal operator is represented by $\operatorname{blkdiag}(\cdot)$ and the Kronecker product is denoted by $\otimes$.
The $z$-transform of a time-domain sequence $\{ x_{k} \}: = (x_0 , x_1 , \ldots)$ is defined by $\hat{x} (z): = \sum_{k = 0}^{\infty} x_{k} z^{-k}$.
We use $\li{x_*}$ to denote a reference point of $x$, or a fixed point when an operator is acting on $x$.

\subsection{Convex analysis}
The general form of a convex optimization problem is given by 
\begin{align}\label{eq: constrained convex optimization problem}
    & \underset{y \in \Omega \subseteq \mathbb{R}^{d}}{\text{minimize}} ~ f(y)
\end{align}
where $f:\mathbb{R}^{d} \to \mathbb{R}$ is the objective function, assumed to be convex, and $\Omega$ is the feasibility set, assumed to be closed and convex. Problem \eqref{eq: constrained convex optimization problem} includes the unconstrained optimization as a special case when $\Omega = \mathbb{R}^{d}$.

\mml{In this work, we consider objective function $f \in S(m, L)$, where $S (m, L)$ denotes the class of differentiable, \textit{$m$-strongly convex}, and \textit{$L$-Lipschitz smooth} functions with $L \geq m > 0$. Specifically, for all $x,~ y \in \mathbb{R}^{d}$, the following condition holds,}
\begin{align*}
    m \|x - y \|^2_2 \leq \left( \nabla f(x) - \nabla f (y)\right)^{\!\top} \!(x - y) \leq L \| x - y\|^2_2.
\end{align*}
Moreover, we consider optimization algorithms with \textit{fixed stepsizes} in this work.

A map $\phi : \mathbb{R}^d \to \mathbb{R}^d$, is said to be \textit{sector-bounded} in $[\alpha,  \beta]$ with $\alpha \leq \beta \in \mathbb{R}$, if $\left(\phi (x) - \alpha x \right)^{\!\top} \!\left(\phi (x) - \beta x \right) \leq 0$, for all $x\in \mathbb{R}^{\li{d}}$. It is said to be \li{\textit{slope-restricted}} in $[\alpha,  \beta]$ with $\alpha \leq \beta \in \mathbb{R}$, if $\left( \phi(x) - \phi(y) - \alpha (x - y) \right)^{\!\top} \! \left( \phi(x) - \phi(y) - \beta (x - y) \right) \leq 0$, for all $x, y \in \mathbb{R}^{\li{d}}$. When $\phi (0) = 0$, a slope-restricted $\phi$ is also sector-bounded in the same sector, but the converse is not generally true.
\mml{It is worth noting that $\nabla f$ with $f \in S(m, L)$ is slope-restricted in $[m, L]$.}

Given a closed and convex set $\Omega$, its \textit{normal cone} at a point $x\in \Omega$ is defined by
$N_{\Omega} (x) = \left\{ v: v^{\top} ( y - x ) \leq 0, ~ \forall  y \in \Omega \right\}$.
\ml{The optimal solution $y^{\textup{opt}}$ to problem \eqref{eq: constrained convex optimization problem} is given by \cite{ruszczynski2011nonlinear},
\begin{align}\label{eq: optimal solution to the constrained problem}
    - \nabla f(y^{\textup{opt}}) \in N_{\Omega} (y^{\textup{opt}}).
\end{align}
The projection of a point $x$ onto a closed convex set $\Omega$ is defined by 
    $
    \Pi_{\Omega} (x) = \underset{y \in \Omega}{\operatorname{argmin}} \| x - y \|_2^2.
    $
A projection is \textit{non-expansive}, i.e.,
\begin{align*}
    \left\| \Pi_{\Omega} (x) - \Pi_{\Omega} (y) \right\|_{2} \leq  \| x - y\|_2,~\text{for all~} x, y.
\end{align*}
The normal cone $N_{\Omega} (x)$ for $x \in \Omega$ can also be written as
$N_{\Omega} (x) = \left\{ v : \Pi_{\Omega} (x + v) = x\right\}$.}
We also define a projection in the weighted inner product space,
\begin{align*}
\Pi_{\Omega}^{\li{P}} (x) = \underset{y \in \Omega}{\operatorname{argmin}} \| x - y \|_{\li{P}}^{2}
\end{align*}
where $\| x\|_{\li{P}}:=  \sqrt{x^{\top} \li{P} x}$ is a $\li{P}$-weighted norm for positive definite \li{matrix} $\li{P}$.

\subsection{Exponential stability and linear convergence}\label{sec: Exponential stability and linear convergence}
Let us review the concepts of linear convergence \cite{ortega2000iterative}, contraction mapping, and exponential stability of discrete-time systems by Lyapunov theory \cite{Khalil2002,bof2018lyapunov}.

Given a sequence $\{ x_{k} \}$ that converges to a reference point $x_{\li{*}}$ in some norm $\| \cdot \|$.
The convergence is said to be \textit{Q-linear} with rate $\rho$ if there exists a constant $\rho \in (0, 1)$ such that $\| x_{k+1} - x_{\li{*}} \| \leq \rho \|x_{k} - x_{\li{*}} \|$, for all $k \geq 0$. It is said to be \textit{R-linear} with rate $\rho$ if there exist constants $\rho \in (0, 1)$ and $c > 0$ such that $\| x_{k} - x_{\li{*}} \| \leq c \rho^{k}$, for all $k \geq 0$.
\begin{lemma}[\mll{\hspace{1sp}\cite{bof2018lyapunov}}]\label{lem: exponential stability}
Consider the non-autonomous discrete-time system
\begin{align*}
    x_{k+1} = f(k, x_{k}), \quad f(k, 0) = 0, ~ \forall k \geq 0
\end{align*}
where $f: \mathbb{N} \times \mathbb{R}^{n} \to \mathbb{R}^{n}$ is Lipschitz continuous with respect to the second argument. 
If there exists $V : \mathbb{N} \times \mathbb{R}^{n} \to \mathbb{R}$, \mll{and some positive constants $c_1$, $c_2$, $c_3$} such that
\begin{align*}
    c_1 \| x_{k} \|^2 \leq V (k, x_{k}) \leq c_2 \| x_{k} \|^2\\
    V (k+1, x_{k+1} )  - V (k, x_{k}) \leq - c_{3} \| x_{k} \|^2
\end{align*}
for all $k \geq 0$, the equilibrium point $x_{\li{*}} = 0$ is exponentially stable, i.e., there exist positive constants $c$, $\lambda$, such that $\|x_{k}\| \leq c \| x_{0}\| e^{- \lambda k}$, for all $k \geq 0$.
\end{lemma}
It is more common to use \textit{linear convergence} to describe algorithms in the optimization community, though it is equivalent to exponential convergence since $e^{-\lambda k} = \rho^{k}$, with $\rho = e^{-\lambda} \in (0, 1)$. From Lemma~\ref{lem: exponential stability} we have 
\begin{align*}
     & V (k+1, x_{k+1} ) \leq V (k, x_{k}) - c_{3} \| x_{k} \|^2\\
    \leq & 
    \left( 1 - \mll{\frac{c_{3}}{c_{2}}} \right) V (k, x_{k})
    : = \rho^2 V ( k, x_{k} )
\end{align*}
\mll{where $\rho \in [0, 1)$ since $c_3 \leq c_2$ follows from $c_{3} \| x_{k} \|^2 \leq V(k, x_{k})$ by the positivity of $ V (k+1, x_{k+1})$.}
Then,
\begin{align*}
    \| x_{k} \|^2 \leq \frac{1}{ c_{1} } V ( k, x_{k} ) \leq \frac{1}{ c_{1} } 
 \rho^{2k} V (0, x_{0}) \leq \frac{c_{2}}{c_{1}} \rho^{2k} \| x_{0} \|^2.
\end{align*}
We thus have
$
    \| x_{k} \| \leq \sqrt{\frac{c_{2}}{c_{1}}} \rho^{k} \| x_{0} \| : = c \rho^{k},
$
meaning that the sequence $\{ x_{k} \}$ converges R-linearly to $x_{\li{*}} = 0$ with rate $\rho$.

The Q-linear and R-linear convergence can be related via Lyapunov theory.
Specifically, suppose that there exists a positive definite quadratic Lyapunov function $V ( x_{k} - x_{\li{*}} ) = \left( x_{k} - x_{\li{*}}\right)^{\! \top} \! P \left( x_{k} - x_{\li{*}}\right)$ such that $V ( x_{k} - x_{\li{*}}) \leq \rho^{2} V (x_{k-1} - x_{\li{*}})$. Then, the sequence $\{ x_{k} \}$ converges R-linearly to $x_{\li{*}}$ with rate $\rho$. Let us define $T$ such that $P = T^{\top} T$, and a linear transformation $\tilde{x} = T x$. Using similar arguments as above with $c_1 = c_2 = 1$, it follows that the sequence $\{ \tilde{x}_{k} \}$ converges Q-linearly to $\tilde{x}_{\li{*}} := T x_{\li{*}}$ with the same rate $\rho$.

\subsection{Lur'e form of first-order algorithms}\label{sec: Lur'e form of first-order algorithms}
First-order discrete-time algorithms for optimization problem \eqref{eq: constrained convex optimization problem}, including gradient descent and Nesterov's methods, can be reformulated into a Lur'e system, where a discrete-time linear time-invariant (LTI) system is interconnected with a nonlinear operator \cite{lessard2016analysis}. The LTI system is given by
\begin{align}\label{eq: LTI system}
\begin{aligned}
    \xi_{k + 1} =& A \xi_{k} + B u_{k}, ~~ \xi_{0} \li{\in \mathbb{R}^{nd}},\\
    y_{k} =& C \xi_{k} + D u_{k}
\end{aligned}
\end{align}
where \li{$A \in \mathbb{R}^{n d \times n d}$,  $B \in \mathbb{R}^{n d \times d}$,  $C\in \mathbb{R}^{ d \times n d}$, $D\in \mathbb{R}^{ d \times d}$, and} the input $u_{k} \in \mathbb{R}^{d}$ is given by the gradient of the output $y_{k} \in \mathbb{R}^{d}$,
\begin{align}\label{eq: gradient input to u_k}
    u_{k} = \nabla f(y_{k}).
\end{align}
It is assumed that $D = 0$ to avoid an algebraic loop, which is also the case for most algorithms \cite{lessard2016analysis, scherer2023optimization}.
It is worth mentioning that $D$ can be non-zero and the nonlinearity can be more than a gradient nonlinearity in \eqref{eq: gradient input to u_k}, when a composition of operators is involved, such as the mirror descent algorithm \cite{li2025convergence} and projection dynamics \cite{lessard2016analysis, lessard2022analysis}.

According to \cite{scherer2021convex}, such a form of first-order algorithms in \eqref{eq: LTI system} can be rewritten as
\begin{align}\label{eq: state-space system observable form}
\begin{bmatrix}
\begin{array}{c}
\xi_{k+1} \\ \hline y_{k}
\end{array}
\end{bmatrix}
= 
\begin{bmatrix}
\begin{array}{c}
\xi^{(1)}_{k+1} \\ \xi^{(2)}_{k+1} \\ \hline y_{k}
\end{array}
\end{bmatrix}
= 
    \begin{bmatrix}
    \begin{array}{cc|c}
         I_d & 0 & I_d\\
         \mll{A_{21}} & \mll{A_{22}} & 0\\
         \hline
         C_{1} & C_{2} & 0
    \end{array}
    \end{bmatrix}
    \begin{bmatrix}
        \begin{array}{c}
             \xi^{(1)}_{k} \\ \xi^{(2)}_{k} \\ \hline u_{k}
        \end{array}
    \end{bmatrix}
\end{align}
where $\xi^{(1)}_{k} \in \mathbb{R}^{d}$, $\xi^{(2)}_{k} \in \mathbb{R}^{(n-1)d}$, for some $n > 1$.
Moreover, because the strong convexity and Lipschitz continuity parameters $m$ and $L$ are scalars independent of the dimension $d$ for any convex objective function $f(x)$ in \eqref{eq: constrained convex optimization problem}, all the matrices involved in \eqref{eq: state-space system observable form} admit the so-called dimensionality reduction \cite[Sec.~4.2]{lessard2016analysis}, i.e., they can be rewritten as a Kronecker product form $(\cdot) \otimes I_d$.
\li{In this work, delays in the input--output channel are not considered, that is, $ CB \neq 0$. Moreover, we focus on algorithms in which the gradient input acts in a descent direction. The following assumption summarizes the class of algorithms considered in this paper.
\begin{assumption}\label{assumption on CB}
    The algorithmic system in \eqref{eq: state-space system observable form} is of relative degree one, and satisfies $C B = C_1 = c_1 I_d \prec 0$.
\end{assumption}
\begin{remark}
    The condition $CB \prec 0$ implies that the gradient input enters the decision variable in a negative direction, which is satisfied by gradient-descent-type algorithms, including gradient descent, Nesterov's accelerated method, and the triple momentum method \cite{lessard2016analysis,van2017fastest}.
    This assumption is also necessary for the present framework in the sense that the direct projection step on the decision variable (see \eqref{alg Euclidean proj} later) does not even yield a fixed point in general if $CB \succ 0$.
    While cases with $C B \succeq 0$ may be addressed using analogous state-space representations, they are left for future work.
\end{remark}
}

Note that we could perform a linear transformation on system \eqref{eq: state-space system observable form}, given by
\begin{align*}
   \tilde{\xi}_{k} = \begin{bmatrix}
       C_{1} & C_{2} \\
       0 & I
   \end{bmatrix}
   \xi_{k}
\end{align*}
such that the output, \li{which serves as the decision variable}, becomes a part of the state vector, while the input is injected exclusively into this part, that is
\begin{align}\label{eq: state-space form with state output for unconstrained optimization}
\begin{bmatrix}
\begin{array}{c}
\tilde{\xi}_{k+1} \\ \hline y_{k}
\end{array}
\end{bmatrix}
=
\begin{bmatrix}
\begin{array}{c}
y_{k+1} \\ \xi^{(2)}_{k+1} \\ \hline y_{k}
\end{array}
\end{bmatrix}
= &
    \begin{bmatrix}
    \begin{array}{cc|c}
         \li{\tilde{A}_{11}} & \li{\tilde{A}_{12}} & \ml{C_{1}} \\
         \li{\tilde{A}_{21}} & \li{\tilde{A}_{22}} & 0\\
         \hline
         I_d & 0 & 0
    \end{array}
    \end{bmatrix}
    \begin{bmatrix}
        \begin{array}{c}
        y_{k} \\ \xi^{(2)}_{k} \\ \hline u_{k}
        \end{array}
    \end{bmatrix} \nonumber\\
    := &
\begin{bmatrix}
    \begin{array}{c|c}
    \tilde{A} & \tilde{B} \\
         \hline
         \tilde{C} & \tilde{D}
    \end{array}
\end{bmatrix}
    \begin{bmatrix}
        \begin{array}{c}
        y_{k} \\ \xi^{(2)}_{k} \\ \hline u_{k}
        \end{array}
    \end{bmatrix}.
\end{align}
It can be observed \ml{from the previous discussion} that $\li{\tilde{A}_{11}} = \li{\tilde{a}_{11}} I_{d}$, $\ml{C_{1} = c_{1}}  I_{d}$ for some scalars $\tilde{a}_{11}$, and $\ml{c_{1}}$.
\mll{\begin{remark}
    The transformation \li{from \eqref{eq: state-space system observable form} to \eqref{eq: state-space form with state output for unconstrained optimization}} is introduced to address two technical issues: (i)~ensuring that \li{the projection step can be directly applied to the decision variable $y$ while preserving optimality under constraints}, and (ii)~establishing convergence with the same rate bound as in the unconstrained case. In the transformed coordinates, the decision variable $y$ becomes part of the system state, so the Lyapunov function remains quadratic in the state while the projection is applied directly to $y$.
\end{remark}}

We will start from this form to construct projected algorithms and show their convergence in Section~\ref{Main results}.

\subsection{IQC-based analysis}
The analysis of system \eqref{eq: gradient input to u_k}, \eqref{eq: state-space form with state output for unconstrained optimization} can be carried out in the framework of integral quadratic constraints (IQCs) \cite{megretski1997system,jonsson2001lecture}.
This subsection provides a brief introduction to the time-domain IQC-based analysis of Lur'e systems \cite{seiler2014stability}.

Denote by $\hat{u} (z)$, $\hat{y} (z)$ the z-transforms for \li{sequences $\left\{ u_{k}\right\}$, $\left\{ y_{k}\right\}$}, respectively. The IQC in the frequency domain is given by 
\begin{align}\label{eq: frequency domain discrete-time IQC}
    \li{\int_{|z| = 1}
    \begin{bmatrix}
        \hat{y} ( z^{-1} ) \\ \hat{u} (  z^{-1}  )
    \end{bmatrix}^{\top}
    \Pi ( z )
    \begin{bmatrix}
        \hat{y} (z ) \\ \hat{u} ( z )
    \end{bmatrix}
    d z
    \geq 0}
\end{align}
where $\Pi (\li{z})$ is a bounded Hermitian matrix characterizing properties that the input and output pair should satisfy.
Suppose there exists such a bounded Hermitian matrix for \eqref{eq: frequency domain discrete-time IQC}, then it can be factorized as 
\begin{align*}
    \Pi (\li{z}) = \Psi^{\top} (\li{ z^{-1} }) M \Psi ( \li{z} )
\end{align*}
with a constant symmetric matrix $M \in \mathbb{R}^{qd \times qd}$, for some $q \geq 2$. 
%
%
\ml{Note that $\Psi (\li{z})$ above can be viewed as the \li{transfer matrix} of an auxiliary \li{stable} linear system, i.e., }
\begin{align}\label{eq: filter system for nonlinearity}
\Psi: ~
\begin{cases}
    \zeta_{k+1} = A_{\Psi} \zeta_{k} + B_{\Psi}^{y} y_{k} + B_{\Psi}^{u} u_{k}, ~\zeta_0 \in \mathbb{R}^{ p d }\\
    h_{k} = C_{\Psi} \zeta_{k} + D_{\Psi}^{y} y_{k} + D_{\Psi}^{u} u_{k},
    \end{cases}
\end{align}
where $\zeta_k\in\mathbb{R}^{pd}$ is the internal state, with $p\in\mathbb{N}$, and $(\zeta_{\li{*}}, h_{\li{*}})$ is the unique fixed point of the system \li{for any choice of $(y_{\li{*}}, u_{\li{*}})$}.
\ml{Then,} the frequency-domain IQC \eqref{eq: frequency domain discrete-time IQC} can be transformed into the time-domain inequality by the Plancherel theorem \li{\cite{rudin1987real}},
\begin{align}\label{eq: time domain discrete-time IQC}
    \sum_{k = 0}^{\infty} \left( h_{k} - h_{\li{*}} \right)^{\! \top} \! M \left( h_{k} - h_{\li{*}} \right) \geq 0.
\end{align}
\li{In this paper, to obtain a stepwise quadratic Lyapunov function, we focus on the case where the multiplier admits a truncated hard IQC, namely,}
\begin{align}\label{eq: truncated hard IQC}
    \li{\sum_{k = t}^{T} \left( h_{k} - h_{*} \right)^{\! \top} M \left( h_{k} - h_{*} \right) \geq 0, \text{~for all~} T \geq t \geq 0.}
\end{align}
\li{By setting $T = t$, or by lifting the samples over a finite window to represent higher-order multipliers~\cite{van2022absolute}, the truncated hard IQC \eqref{eq: truncated hard IQC} directly yields a pointwise IQC and hence a stepwise quadratic Lyapunov function.}
Combining systems \eqref{eq: state-space form with state output for unconstrained optimization} and \eqref{eq: filter system for nonlinearity} by eliminating $y_{k}$, we have
\begin{align}\label{eq: augmented system with IQC}
\begin{aligned}
    x_{k + 1} 
    = \hat{A} x_{k} + \hat{B} u_{k}\\
    h_{k} = \hat{C} x_{k} + \hat{D} u_{k}
\end{aligned}
\end{align}
where $x_{k} = (\tilde{\xi}_{k}, \zeta_{k})$, and
\begin{subequations}\label{eq: IQC augmented system matrices}
\begin{align}
& \hat{A} = 
\begin{bmatrix}
    \tilde{A} & 0\\
    B_{\Psi}^{y} \tilde{C} & A_{\Psi}
\end{bmatrix}
\in \mathbb{R}^{(n + p) d \times (n+p) d }, \label{eq: A_hat}\\
& \hat{B} = 
\begin{bmatrix}
    \tilde{B} \\ B_{\Psi}^{u} + B_{\Psi}^{y}  \tilde{D} 
\end{bmatrix}
\in \mathbb{R}^{ (n + p)d \times d},\\
& \hat{C} = 
\begin{bmatrix}
    D_{\Psi}^{y} \tilde{C} & C_{\Psi}
\end{bmatrix}
\in \mathbb{R}^{ qd \times ( p + n) d},\\
& \hat{D} = 
    D_{\Psi}^{u} + D_{\Psi}^{y} \tilde{D} \in \mathbb{R}^{qd \times d}.
\end{align}
\end{subequations}
Then, system \eqref{eq: augmented system with IQC} in conjunction with \eqref{eq: time domain discrete-time IQC}, \icl{can be} used to construct dissipation inequalities with \icl{a} corresponding linear matrix inequality (LMI) interpretation, through which stability of the closed-loop system can be deduced \cite{seiler2014stability}.
\icl{In particular,} the output $y_{k}$ of the closed-loop system \eqref{eq: LTI system}, \eqref{eq: gradient input to u_k} converges R-linearly to $y_{\li{*}}$, \ml{which is} the optimal solution to the unconstrained version of problem \eqref{eq: constrained convex optimization problem}, if there exist $\rho \in (0, 1)$, $M$ in \eqref{eq: truncated hard IQC}, and $\mathbf{P} = \mathbf{P}^{\top} \succ 0$
such that 
\begin{align}\label{eq: LMI from IQC}
    \begin{bmatrix}
        \hat{A}^{\top} \mathbf{P} \hat{A} - \rho^2 \mathbf{P} & \hat{A}^{\top} \mathbf{P} \hat{B}\\
        \hat{B}^{\top} \mathbf{P} \hat{A} & \hat{B}^{\top} \mathbf{P} \hat{B}
    \end{bmatrix}
    +
    \begin{bmatrix}
        \hat{C}^{\top} \\ \hat{D}^{\top}
    \end{bmatrix}
    M
    \begin{bmatrix}
        \hat{C} & \hat{D}
    \end{bmatrix}
    \preceq 0.
\end{align}
\li{Let $V(e) = e^{\top} \mathbf{P}e$.} Multiplying \eqref{eq: LMI from IQC} on the left and right by 
$\begin{bmatrix}
     \left( x_{k} - x_{\li{*}} \right)^{\! \top} & \left( u_{k} - u_{\li{*}}\right)^{\!\top}
\end{bmatrix}$
and its transpose, respectively, we obtain
\begin{align}\label{eq: quadratic Lyapunov function}
    V (x_{k+1} \li{- x_{*+1}}) \leq \rho^2 V (x_{k} \li{- x_*}), \quad \forall k \geq 0
\end{align}
\li{for any reference point $(x_{*+1},x_*,h_*, y_*,u_*)$ satisfying \eqref{eq: augmented system with IQC} where $x_{* + 1}$ denotes the next state of $x_*$}, meaning that the quadratic Lyapunov function is exponentially decreasing and linear convergence is thus guaranteed.
\mml{As $\nabla f$ is slope-restricted in $[m, L]$, a rich class of O'Shea–-Zames–-Falb (OZF) multipliers can be employed to generate IQCs \eqref{eq: frequency domain discrete-time IQC}, characterizing the slope-restricted nonlinearity in \eqref{eq: gradient input to u_k}. Subsequently, the LMI in \eqref{eq: LMI from IQC} can be used to obtain tight convergence rate bounds \cite{lessard2016analysis}.}

\ml{In this subsection, we have provided} a brief overview of the connection between IQC-based analysis and quadratic Lyapunov functions, with the latter being the primary focus of this paper.
More details on the IQC analysis can be found in \cite{rantzer1996kalman, jonsson2001lecture,fu2005integral,seiler2014stability,carrasco2019convex}.

\section{Main results}\label{Main results}

\subsection{Pointwise IQCs and Quadratic Lyapunov Function}

\mml{We adopt the following assumption to guarantee the linear convergence for the unconstrained dynamics.}
\begin{assumption}\label{assumption on stepwise quadratic Lyapunov function}
For the unconstrained algorithm \eqref{eq: gradient input to u_k}, 
\eqref{eq: state-space form with state output for unconstrained optimization}, 
there exist $M$ in \eqref{eq: truncated hard IQC}, $\rho \in (0,1)$ and $\mathbf P=\mathbf P^\top\succ 0$ such that 
\eqref{eq: LMI from IQC} is satisfied. 
\end{assumption}
The matrix $\mathbf P$ \li{in Assumption~\ref{assumption on stepwise quadratic Lyapunov function}} can be partitioned as
\begin{align}\label{eq: partition of P}
    \mathbf{P} = \begin{bmatrix}
        \mathbf{P}_{11} & \mathbf{P}_{12}\\
        \mathbf{P}_{12}^{\top} & \mathbf{P}_{22}
    \end{bmatrix}
    = \begin{bmatrix}
        {P}_{11} & {P}_{12}\\
        {P}_{12}^{\top} & {P}_{22}
    \end{bmatrix} \otimes I_{d}, \quad P_{11} \in \mathbb{R}.
    \end{align}
\li{This structure} naturally holds because the matrices in \eqref{eq: state-space form with state output for unconstrained optimization}, and consequently in \eqref{eq: LMI from IQC}, can be expressed in a Kronecker form that is independent of the variable dimension $d$, as discussed in Section~\ref{sec: Lur'e form of first-order algorithms}.

Then, the unconstrained optimization algorithm \eqref{eq: gradient input to u_k}, \eqref{eq: state-space form with state output for unconstrained optimization} is linearly convergent with rate $\rho$, and there exists a quadratic Lyapunov function \li{$V (e) = e^\top \mathbf{P} e$}, such that \li{\eqref{eq: quadratic Lyapunov function} holds}.
\begin{remark}
    \li{Assumption~\ref{assumption on stepwise quadratic Lyapunov function} is closely related to many hard-IQC-based analyses such as those in \cite{lessard2016analysis,michalowsky2021robust}.
    In particular, while hard IQCs are typically written on truncated horizons starting from the initial time (see, e.g., \cite[Lemma~8, Lemma~10]{lessard2016analysis}), the corresponding moving-horizon version \eqref{eq: truncated hard IQC} can be verified directly in these cases. This is not automatic in general, however, especially in the presence of specific initial-condition terms (see, e.g., \cite{scherer2023optimization}).
    The assumption is also well aligned with lifting-based constructions, where samples over a finite time window are lifted, and higher-order multipliers are imposed on the lifted variables, which has been shown to yield tight certificates in practice \cite{van2022absolute}.}
    We also note that recent works have investigated the necessity of OZF multipliers for robust stability of related Lur'e systems \cite{carrasco2016zames,su2023necessity}.
\end{remark}

\subsection{Projected algorithm in the norm induced by the Lyapunov matrix}
This subsection presents \li{the projected algorithm} for the constrained optimization problem \eqref{eq: constrained convex optimization problem}, derived from the state-space formulation in \eqref{eq: state-space form with state output for unconstrained optimization}.

\li{A straightforward projected algorithm is} to perform a projection directly on the output $y_{k}$ to the feasibility set \li{$\Omega$, which gives}
%
\begin{subequations}\label{alg Euclidean proj}
\begin{align}
\begin{bmatrix}
y_{k+\frac{1}{2}} \\ \xi^{(2)}_{k+1}
\end{bmatrix}
= &
\tilde{A}
    \begin{bmatrix}
         y_{k} \\ \xi^{(2)}_{k}
    \end{bmatrix}
    + \tilde{B} \li{\nabla f ( y_{k})} \label{eq: state-space form of projected algorithm}\\
y_{k + 1} = & \Pi_{\Omega} \left(y_{k + \frac{1}{2}} \right) \label{eq: projection of output}
\end{align}
\end{subequations}
\li{where the state} $y_{k + \frac{1}{2}}$ is an intermediate state, 
$(\tilde{A}, \tilde{B})$ are given in \eqref{eq: state-space form with state output for unconstrained optimization}\li{,}
\li{and} the projection operator is $\Pi_{\Omega} (x) = \underset{z \in \Omega}{\operatorname{argmin}} \| z - x \|_{2}^{2}$.

\begin{figure}
    \centering
    \includegraphics[width=1\linewidth]{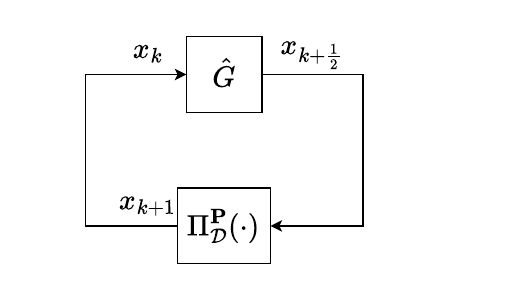}
    \caption{\mll{The iteration of Algorithm~\ref{alg:2}, where the nonlinear operator $\hat{G}$ represents the closed-loop system of \eqref{eq: gradient input to u_k} and \eqref{eq: state space of the projected algorithm with augmented system}, $\mathbf{P}$ is given by Assumption~\ref{assumption on stepwise quadratic Lyapunov function}, and $\mathcal{D} = \Omega \times \mathbb{R}^{(n+p-1)d}$.}}
    \label{fig: projected algorithm 2}
\end{figure}

\li{
We can rewrite \eqref{alg Euclidean proj} as
\[
\tilde{\xi}_{k+1} = \Pi_{\tilde{\Omega}}
\left(
\tilde{A}
\tilde{\xi}_k
+ \tilde{B} \nabla f ( \tilde{C} \tilde{\xi}_{k} )
\right), ~ \tilde{\Omega} = \Omega \times \mathbb{R}^{(n-1)d}.
\]
Thus, \eqref{alg Euclidean proj} corresponds to applying a Euclidean projection to an unconstrained algorithm whose error decays exponentially in the $\mathbf P$-norm. This mismatch does not guarantee the decrease of the error bound at the same rate.
}

\mll{
\li{Motivated by this mismatch,} we propose a projected algorithm obtained by the projection onto the weighted inner product space induced by the Lyapunov matrix $\mathbf P$, which leads to Algorithm~\ref{alg:2}.
}
\begin{algorithm}
\caption{Projected algorithm on the $\mathbf{P}$ norm}\label{alg:2}
\mll{\begin{algorithmic}
\State
\textbf{Initialization}: $\tilde{\xi}_{0} \in \mathbb{R}^{nd}$, $\zeta_{0} 
\in \mathbb{R}^{ p d }$, and $x_{0} = (\tilde{\xi}_{0}, \zeta_{0})$.
\State
\textbf{Iteration}:
\begin{align}\label{eq: state space of the projected algorithm with augmented system}
    x_{k + \frac{1}{2}} 
    =
    \begin{bmatrix}
        y_{k+ \frac{1}{2} } \\ {\xi}_{k + \frac{1}{2}}^{(2)} \\ \zeta_{k + \frac{1}{2}}
    \end{bmatrix}
    = \hat{A} x_{k} + \hat{B} u_{k}
\end{align}
where $x_{k + \frac{1}{2}}$ is an intermediate state and $(\hat{A}, \hat{B}
)$ are defined in \eqref{eq: IQC augmented system matrices}.\\
The input $u_{k}$ is given by \eqref{eq: gradient input to u_k} with $y_k = x_k^{(1)}$, i.e., the first part of the system state $x_{k}$.\\
The update $x_{k + 1}$ is given by
\begin{align}\label{eq: projected algorithm in the P norm}
    x_{k + 1} = \Pi_{\mathcal{D}}^{\mathbf{P}} (x_{k+\frac{1}{2}}) = \underset{z \in \mathcal{D} }{\operatorname{argmin}} \| z - x_{k+\frac{1}{2}} \|_{\mathbf{P}}^{2}
\end{align}
where $\Pi_{\mathcal{D}}^{\mathbf{P}} (\cdot)$ is the projection in the weighted norm with $\mathcal{D} = \Omega \times \mathbb{R}^{(n+p-1)d} \in \mathbb{R}^{(n+p)d}$ and $\mathbf{P}$ given by Assumption~\ref{assumption on stepwise quadratic Lyapunov function}.\\
The output is $y_{k+1} = x^{(1)}_{k+1}$.
\end{algorithmic}}
\end{algorithm}

\mll{The iteration of \eqref{eq: gradient input to u_k}, \eqref{eq: state space of the projected algorithm with augmented system}, and \eqref{eq: projected algorithm in the P norm} in Algorithm~\ref{alg:2} can be represented by the diagram in \li{Fig.~\ref{fig: projected algorithm 2}}.}

\li{
To analyze the fixed point of Algorithm~\ref{alg:2}, we introduce the following lemmas.}
\li{
\begin{lemma}\label{lem: property of the system matrix A}
Consider the system matrix $\tilde A$ in \eqref{eq: state-space form with state output for unconstrained optimization}. Under Assumption~\ref{assumption on CB}, $\det \left( I - \tilde{A}_{22}\right) > 0$, and 
\begin{align}\label{eq: property of the system matrix A}
\tilde{A}_{11} + \tilde{A}_{12} \left( I - \tilde{A}_{22} \right)^{-1} \tilde{A}_{21} = I_d.
\end{align}
\end{lemma}
\begin{proof}
    See Appendix~\ref{appendix proof of Lemma property of the system matrix A}.
\end{proof}
It can be observed that $\hat{A}$ in \eqref{eq: IQC augmented system matrices} also satisfies similar properties given by Lemma~\ref{lem: property of the system matrix A} since $A_{\Psi}$ is Schur stable.
}

\li{
\begin{lemma}\label{lem: Schur of transformed A_22}
Suppose matrix $A$ is Schur stable and there exists a positive definite $P \succ 0$ such that 
\begin{align}
    A ^\top {P} A \preceq \rho^2 {P} 
\end{align}
for some $\rho \in [0, 1)$. 
Partition 
$A = \begin{bmatrix}
    A_{11} & A_{12}\\ A_{21} & A_{22}
\end{bmatrix}$ and $P = \begin{bmatrix}
    P_{11} & P_{12}\\ P_{12}^\top & P_{22}
\end{bmatrix}$ with compatible dimensions.
Then, the matrices
$A_{11} - A_{12} P_{22}^{-1} P_{12}^\top$
and
${A}_{22} + {P}_{22}^{-1} {P}_{12}^\top {A}_{12}$
are also Schur stable.
\end{lemma}
\begin{proof}
    See Appendix~\ref{appendix proof of Lemma Schur of transformed A_22}.
\end{proof}
}
\mll{
We show by the following theorem that the fixed-point of \eqref{eq: gradient input to u_k}, \eqref{eq: state space of the projected algorithm with augmented system}, \eqref{eq: projected algorithm in the P norm} in Algorithm~\ref{alg:2} provides the same optimal solution to the constrained problem.
\begin{theorem}\label{thm: fixed point}
    \li{Under Assumptions~\ref{assumption on CB} and \ref{assumption on stepwise quadratic Lyapunov function},}
    there exists a unique fixed point $(u_{\li{*}}, y_{\li{*}}, x_{\li{*}})$ to the dynamics \eqref{eq: gradient input to u_k}, \eqref{eq: state space of the projected algorithm with augmented system}, \eqref{eq: projected algorithm in the P norm}
    such that $y_{\li{*}}$ is the optimal solution to problem \eqref{eq: constrained convex optimization problem}.
\end{theorem}
}
\begin{proof}
\mll{
    First, note that $y$ is the first part of the system state $x$,
    we can partition the state as $x = (x^{(1)}, x^{(2)})$, with $x^{(1)} = y = \tilde{\xi}^{(1)}$ and $x^{(2)} = (\xi^{(2)}, \zeta)$.
    We show that the projection effect on $y$ in \eqref{eq: projected algorithm in the P norm} is equivalent to that in \eqref{eq: projection of output}. Specifically, from \eqref{eq: projected algorithm in the P norm}, we have
    \begin{align*}
        x_{k+1} = & \underset{z\in \mathbb{R}^{(n+p)d}}{\operatorname{argmin}} \left\| z - x_{k + \frac{1}{2}} \right\|_{\mathbf{P}}^{2}\\
        = & \underset{
                \substack{
                    z^{(1)} \in \Omega,\\
                    z^{(2)} \in \mathbb{R}^{(n+p-1)d}
                    }
            }{\operatorname{argmin}} \left\{ \left( z^{(1)} - x^{(1)}_{k + \frac{1}{2}} \right)^{\!\top} \mathbf{P}_{11} \left( z^{(1)} - x^{(1)}_{k + \frac{1}{2}} \right) \right.\\[-10pt]
        & \qquad\qquad\quad + 2 \left( z^{(1)} - x^{(1)}_{k + \frac{1}{2}} \right)^{\!\top} \mathbf{P}_{12} \left( z^{(2)} - x^{(2)}_{k + \frac{1}{2}} \right) \\
        & \qquad\qquad\quad + \left. \left( z^{(2)} - x^{(2)}_{k + \frac{1}{2}} \right)^{\!\top} \mathbf{P}_{22} \left( z^{(2)} - x^{(2)}_{k + \frac{1}{2}} \right) \right\}
    \end{align*}
    where $z = \left(z^{(1)}, z^{(2)}\right)$. The partial derivative with respect to $z^{(2)}$ is zero as it is unconstrained, then
    \begin{align*}
        x^{(2)}_{k + 1}
        = & \underset{ z^{(2)} \in \mathbb{R}^{(n+p-1)d}  }{\operatorname{argmin}} \left\{
         2 \left( z^{(1)} - x^{(1)}_{k + \frac{1}{2}} \right)^{\!\top} \mathbf{P}_{12} \left( z^{(2)} - x^{(2)}_{k + \frac{1}{2}} \right) \right. \nonumber\\
         & \qquad\qquad\quad  + \left. \left( z^{(2)} - x^{(2)}_{k + \frac{1}{2}} \right)^{\!\top} \mathbf{P}_{22} \left( z^{(2)} - x^{(2)}_{k + \frac{1}{2}} \right) \right\} \nonumber\\
        = & x^{(2)}_{k+\frac{1}{2}} - \mathbf{P}_{22}^{-1} \mathbf{P}_{12}^{\top} \left( z^{(1)} - x^{(1)}_{k + \frac{1}{2}}\right).
    \end{align*}
    Substituting $z^{(2)}$ back into the projection, we have 
    \begin{align*}
        & x^{(1)}_{k + 1} \\
        = & \underset{z^{(1)} \in \Omega}{\operatorname{argmin}} \left( z^{(1)} - x^{(1)}_{k + \frac{1}{2}}\right)^{\!\top} \! \left( \mathbf{P}_{11} - \mathbf{P}_{12} \mathbf{P}_{22}^{-1} \mathbf{P}_{12}^{\top} \right) \! \left( z^{(1)} - x^{(1)}_{k + \frac{1}{2}}\right)\\
        = & \underset{z^{(1)} \in \Omega}{\operatorname{argmin}} \left\| z^{(1)} - x^{(1)}_{k + \frac{1}{2}}\right\|^{2}_{2}\\
        =& \Pi_{\Omega} \left( x^{(1)}_{k+\frac{1}{2}} \right)
    \end{align*}
    where the second equality 
    \li{holds in that} the Schur complement $\left( \mathbf{P}_{11} - \mathbf{P}_{12} \mathbf{P}_{22}^{-1} \mathbf{P}_{12}^{\top} \right)$ is positive definite and inherits a Kronecker structure $(\cdot) \otimes I_{d}$ of a scalar.
    Then, the projection effect on $y$ in \eqref{eq: projected algorithm in the P norm} is equivalent to that in \eqref{eq: projection of output}.}
    \mll{Next, let us analyze the fixed point of the dynamics. We partition the system matrix $\hat{A}$ as
    \begin{align*}
        \hat{A} = \begin{bmatrix}
            \hat{A}_{11} & \hat{A}_{12}\\
            \hat{A}_{21} & \hat{A}_{22}
        \end{bmatrix}
    \end{align*}
    where $\hat{A}_{11} \in \mathbb{R}^{d \times d}$ and $\hat{A}_{22} \in \mathbb{R}^{(n+p-1)d \times (n+p-1)d}$.
    The fixed point $\left( y_{\li{*}}, x^{(1)}_{\li{*},\frac{1}{2}}, x^{(2)}_{\li{*}} \right)$ of the projected algorithm should satisfy
    \begin{subequations}\label{eq: fixed point of alg 2}
    \begin{align}
     & x^{(1)}_{\li{*},\frac{1}{2}} = \hat{A}_{11} y_{\li{*}} + \hat{A}_{12} x^{(2)}_{\li{*}} + c_1 \nabla f(y_{\li{*}}) \label{eq: first block of equilibrium}\\
     & \Pi_{\Omega} \left( x^{(1)}_{\li{*},\frac{1}{2}} \right) = y_{\li{*}} \label{eq: fixed point projection}\\
     & x^{(2)}_{\li{*}} = \hat{A}_{21} y_{\li{*}} + \hat{A}_{22} x^{(2)}_{\li{*}} + \li{\hat{B}_2 \nabla f (y_*)} - \mathbf{P}_{22}^{-1} \mathbf{P}_{12}^{\top} \left( y_{\li{*}} - x^{(1)}_{\li{*},\frac{1}{2} }\right) \label{eq: second block of equilibrium}
    \end{align}
    \end{subequations}}\li{
where $\hat{B}_2$ is the associated partitioned block of $\hat{B}$.
As $I - \hat{A}_{22}$ is nonsingular by Lemma~\ref{lem: property of the system matrix A}, $x_{*}^{(2)}$ is then given by
\begin{align*}
\begin{aligned}
x_*^{(2)}
= \big(I - \hat{A}_{22} \big)^{-1} & \Big(
\hat{A}_{21} y_* + \hat{B}_2 \nabla f (y_*) \\
&\qquad
- \mathbf{P}_{22}^{-1} \mathbf{P}_{12}^{\top}
\left(y_* - x_{*,\frac12}^{(1)} \right)
\Big).
\end{aligned}
\end{align*}}
\li{Substituting it into \eqref{eq: first block of equilibrium} yields
\begin{align*}
& \left(\hat A_{11}  + \hat A_{12} \big( I - \hat{A}_{22} \big)^{-1} \hat{A}_{21} \right) y_* - x_{*, \frac12}^{(1)}\\
& + \hat A_{12} \big( I - \hat{A}_{22}\big)^{-1} \hat{B}_2 \nabla f (y_*)
+ c_1\nabla f(y_{\li{*}} )\\
& - \hat A_{12} \big( I - \hat{A}_{22} \big)^{-1} \mathbf{P}_{22}^{-1} \mathbf{P}_{12}^{\top} \left(y_* - x_{*,\frac12}^{(1)} \right)\\
= & \Big( I - \hat A_{12} \big( I - \hat{A}_{22}\big)^{-1} \mathbf{P}_{22}^{-1} \mathbf{P}_{12}^{\top}  \Big)  \left( y_* - x_{*,\frac12}^{(1)} \right) + c_1 \nabla f (y_*)\\
:= & \eta \left( y_* - x_{*,\frac12}^{(1)} \right)  + c_1  \nabla f (y_*) = 0
\end{align*}
where the first equality follows from Lemma~\ref{lem: property of the system matrix A} and the readily verifiable condition $\hat A_{12} \big( I - \hat{A}_{22}\big)^{-1} \hat{B}_2 = 0$, the second equality follows from the Kronecker structure of the matrices.
By the fixed point condition of the projection in \eqref{eq: fixed point projection}, that is, 
\(
x^{(1)}_{\li{*},\frac{1}{2}} - y_{\li{*}} \in N_\Omega(y_{\li{*}}),
\) we can obtain 
\begin{align}\label{eq: fix point with indefinite index}
    c_1 \nabla f (y_{\li{*}}) \in \eta N_{\Omega}\left( y_{\li{*}} \right).
\end{align}
To identify the sign of $\eta$, we first observe that $\left(\hat{A}+ m \hat{B} \begin{bmatrix}
    I_d & 0
\end{bmatrix}\right)$ is Schur stable, since the associated closed-loop system matrix $\tilde{A}+ m \tilde{B}\tilde{C}$ is Schur stable \cite{scherer2021convex}, and the same $\mathbf{P}$ is a Lyapunov matrix for this matrix because it corresponds to the quadratic objective $f (y) = \frac{m}{2} \| y \|^2$, which is included in the considered function class.
Next, we have
\begin{align*}
    & \det \left(I_d - \hat{A}_{12} \left( I - \hat{A}_{22} \right)^{-1} \mathbf{P}_{22}^{-1} \mathbf{P}_{12}^{\top} \right)\\
    = & \det \left(I - \mathbf{P}_{22}^{-1} \mathbf{P}_{12}^{\top} \hat{A}_{12} \left( I - \hat{A}_{22} \right)^{-1} \right)\\
    = & \det \left(  I - \hat{A}_{22} - \mathbf{P}_{22}^{-1} \mathbf{P}_{12}^{\top} \hat{A}_{12} \right) \det \left( I - \hat{A}_{22} \right)^{-1} > 0
\end{align*}
where the first equality follows from Sylvester's determinant theorem \cite[Fact 2.17.4]{bernstein2009matrix}, and the inequality holds as $\det \left(I - \hat{A}_{22} \right) > 0$ by applying Lemma~\ref{lem: property of the system matrix A} to $\hat{A}$, and $\hat{A}_{22} + \mathbf{P}_{22}^{-1} \mathbf{P}_{12}^\top \hat{A}_{12}$ is Schur stable by Lemma~\ref{lem: Schur of transformed A_22} to $\left( \hat{A} + m \hat{B} \begin{bmatrix}
    I_d & 0
\end{bmatrix} \right)$.
Therefore, $\eta > 0$. Moreover, as $c_1 < 0$ by Assumption~\ref{assumption on CB}, condition \eqref{eq: fix point with indefinite index} satisfies the optimality condition for the constrained problem \eqref{eq: constrained convex optimization problem}.} The uniqueness of \li{$y_*$} follows from the strong convexity of $f$ and the uniqueness of the projection onto a closed convex set \cite{ruszczynski2011nonlinear} \li{and $x_*^{(2)}$ is then uniquely determined by \eqref{eq: second block of equilibrium}}.
\end{proof}
\mll{
\begin{remark}
    \li{While algorithms of the form \eqref{alg Euclidean proj} use a direct projection step, the mismatch between the convergence norm and the Euclidean norm may prevent the corresponding convergence rate bounds from being retained.} Algorithm~\ref{alg:2} incorporates augmented states derived from the IQC analysis and introduces a shift in the second state. Such a shift may act like a momentum term, ensuring the same convergence rate bound as the unconstrained case.
    Notice that the update of $x^{(1)}_{k}$ does not \li{depend} on the augmented state $\zeta_{k}$, as shown in \eqref{eq: A_hat}. Thus, in practice, we only need to update $\xi^{(2)}_{k}$ as part of $x^{(2)}_{k}$, with knowledge of the Lyapunov matrix $\mathbf{P}$.
    \li{However, the drawback is that its construction relies on IQCs and on the numerical computation of the Lyapunov matrix \(\mathbf{P}\) (see the example in Section~\ref{subsection: Projected triple momentum algorithm}), which may become numerically involved for ill-conditioned problems. Future work may aim to relax these requirements.}
\end{remark}
}

\subsection{\mll{Projection under \li{linear} transformation}}
Let us review that the operator $\mll{T}^{-1} \Pi_{\Omega} (\mll{T} \cdot )$ with an arbitrary invertible linear transformation $T$ does not preserve the non-expansiveness in general, that is, for any $x$, $z$,
\begin{align}\label{eq: norm inequality under general T}
     & \left\| \mll{T^{-1}} \Pi_{\Omega} (\mll{T} x ) -  \mll{T^{-1}} \Pi_{\Omega} (\mll{T} z ) \right\|_2  \leq \| \mll{T^{-1}} \| \cdot \| \mll{T}\| \cdot \left\| x - z \right\|_2.
\end{align}
However, it is still a projection with a weighted norm, as shown by the following lemma.
\begin{lemma}\label{lem: projection under general transformation}
    Consider the projection operator $\Pi_{\Omega} (\cdot)$ \li{onto} a closed convex set $\Omega \subseteq \mathbb{R}^{d}$, and an arbitrary invertible matrix $T$, then $ \mll{T^{-1}} \Pi_{\Omega} (\mll{T}\cdot )$ is a projection onto $\mll{T^{-1}} \Omega$ with \li{respect to} the \li{weighted} norm 
    $\| \cdot \|_{\mll{T^\top T}}$.
\end{lemma}
\begin{proof}
    First, a projection should be idempotent, i.e., the composition $\left( \mll{T^{-1}} \Pi_{\Omega} (\mll{T} \cdot ) \right) \circ \left( \mll{T^{-1}} \Pi_{\Omega} (\mll{T} \cdot ) \right) = \mll{T}^{-1} \Pi_{\Omega} (\mll{T} \cdot )$. Next, we have
    \begin{align*}
        \Pi_{\Omega} ( \mll{T} {x}) = & \underset{\li{z} \in \Omega}{\operatorname{argmin}} \| z - \mll{T} x \|_2^2\\
        = & \underset{z \in \Omega, \tilde{z} = \mll{T^{-1}} z}{\operatorname{argmin}} \| \mll{T} \tilde{z} - \mll{T} {x} \|_2^2\\
        = & T \cdot \underset{\tilde{z} \in \mll{T^{-1}} \Omega}{\operatorname{argmin}} \| \mll{T} \left(\tilde{z} - {x} \right)\|_2^2\\
        = & T \cdot \underset{\tilde{z} \in \mll{T^{-1}} \Omega }{\operatorname{argmin}} \|  \tilde{z} - {x} \|_{\mll{ T^\top T}}^2\\
        := & T \cdot \Pi_{ \mll{T^{-1}} \Omega}^{T^\top T} (x)
    \end{align*}
    where $\mll{T^{-1}} \Omega  = \left\{ \mll{T^{-1}} z : z \in \Omega\right\}$, and the third equality follows from the fact that the change of variable $\tilde{z} = T^{-1} z$ is bijective.
    Therefore, we have 
    \begin{align*}
    \mll{T^{-1}} \Pi_{\Omega} (\mll{T} x) = \underset{\tilde{z} \in \mll{T^{-1}} \Omega }{\operatorname{argmin}} \| \tilde{z} - x \|_{ T^{\top} T }^2 = \Pi_{ \mll{T^{-1}} \Omega}^{\mll{T^{\top} T} } (x),
    \end{align*}
    which is a projection onto the transformed set $\mll{T^{-1}} \Omega $ with respect to the weighted norm $\|\cdot \|_{\mll{T^{\top} T} }$.
\end{proof}

\subsection{Convergence analysis}

\mll{The convergence of Algorithm \ref{alg:2} is established in the following theorem.}
\begin{theorem}
    \li{Suppose Assumptions~\ref{assumption on CB} and \ref{assumption on stepwise quadratic Lyapunov function} hold,} \ml{so that the unconstrained algorithm given by \eqref{eq: gradient input to u_k}, \eqref{eq: state-space form with state output for unconstrained optimization} is R-\li{linearly} convergent with rate $\rho$}. Then the corresponding
    \mll{Algorithm~\ref{alg:2}}
    is also R-\li{linearly} convergent with the same rate $\rho$. 
\end{theorem}
\begin{proof}
\ml{Let us first look at the unconstrained algorithm \eqref{eq: gradient input to u_k}, \eqref{eq: state-space form with state output for unconstrained optimization}.} Let the quadratic Lyapunov function be $V(\li{e}) = \li{e}^{\top} \mathbf{P}\li{e}$, with $\mathbf{P} = \mathbf{P}^{\top} \succ 0$. Then, \ml{by Assumption~\ref{assumption on stepwise quadratic Lyapunov function} and \eqref{eq: LMI from IQC},} we have \eqref{eq: quadratic Lyapunov function}, \ml{meaning that the sequence $\{x_{k} \}$ converges R-linearly to $x_{\li{*}}$ with rate $\rho$}.
We then introduce a linear transformation $\tilde{x} = \mathbf{T} x$ where $\mathbf{T}$ is an invertible matrix satisfying $\mathbf{T}^{\top} \mathbf{T} = \mathbf{P}$.
By our discussion in Section~\ref{sec: Exponential stability and linear convergence}, the transformed sequence $\tilde{x}_{k} = T x_{k}$ converges Q-linearly to $\tilde{x}_{\li{*}} = \mathbf{T} x_{\li{*}}$ with the same rate, that is, $\| \tilde{x}_{k} - \tilde{x}_{\li{*}} \|_2 \leq \rho \|\tilde{x}_{k - 1} - \tilde{x}_{\li{*}} \|_2$.

\ml{Next, let us look at Algorithm~\ref{alg:2} characterized by 
\eqref{eq: gradient input to u_k}, \eqref{eq: state space of the projected algorithm with augmented system}, \eqref{eq: projected algorithm in the P norm}. Now, \li{$x_{k + \frac{1}{2}}$ represents the unconstrained state update, that is, $x_{k + \frac12} = \hat{A} {x}_{k} + \hat{B} \nabla f ( y_k)$.
Meanwhile, $x_*$ represents the fixed point of the projected dynamics, with an abuse of notation.}
\li{By Theorem~\ref{thm: fixed point}, the fixed point $(u_{\li{*}}, y_{\li{*}}, x_{\li{*}})$  is unique with $y_*$ satisfying \eqref{eq: optimal solution to the constrained problem}, proving optimality.
Then, we have
$x_{*, \frac{1}{2}} = x_{* + 1} = \hat{A} {x}_* + \hat{B} \nabla f ( {y}_*)$, and $
    x_* = \Pi_{\mathcal{D}}^{\mathbf{P}} \left(x_{*, \frac{1}{2}} \right)$.}
We perform the same linear transformation to $x_{k+\frac{1}{2}}$, that is, $\tilde{x}_{\li{k+\frac{1}{2}}} = \mathbf{T} x_{k+\frac{1}{2}}$.}
The control diagram of this transformed dynamics is depicted in \li{Fig~\ref{fig: augmented system transformed}}.
\begin{figure}
    \centering
    \includegraphics[width=1\linewidth]{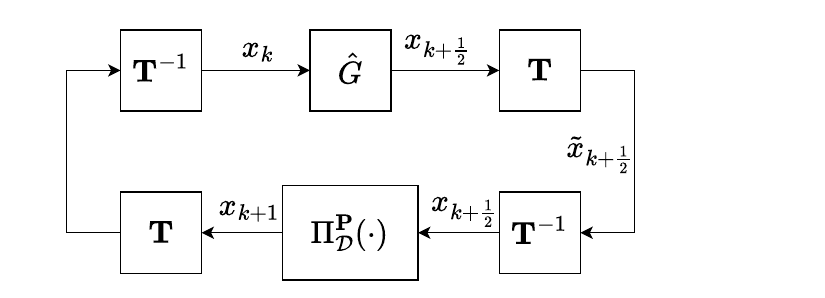}
    \caption{The system diagram under linear transformation, where $\tilde{G}$ represents the closed-loop system of \eqref{eq: augmented system with IQC} and \eqref{eq: gradient input to u_k}, $\Pi_{\mathcal{D}}^{\mathbf{P}} (\cdot)$ is a projection onto the set $\mathcal{D}:= \Omega \times \mathbb{R}^{(p + n - 1)d}$ with the induced norm $\mathbf{P} = \mathbf{T}^{\top} \mathbf{T}$.}
    \label{fig: augmented system transformed}
\end{figure}
\mll{As $\mathbf{P} = \mathbf{T}^{\top} \mathbf{T}$, we have by Lemma~\ref{lem: projection under general transformation} \li{that $\Pi_{\mathcal{D}}^{\mathbf{P}} ( x ) = \mathbf{T}^{-1} \Pi_{ \mathbf{T}\mathcal{D}} (  \mathbf{T} x )$, and thus}
\begin{align*}
    \mathbf{T} \Pi_{\mathcal{D}}^{\mathbf{P}} ( \mathbf{T}^{-1} \tilde{x})
    = \Pi_{\mathbf{T} \mathcal{D}} ( \tilde{x} ).
\end{align*}}
\mll{The transformed operator $\mathbf{T} \Pi_{\mathcal{D}}^{\mathbf{P}} ( \mathbf{T}^{-1} \tilde{x}) = \Pi_{\mathbf{T} \mathcal{D}} (\tilde{x}) $ is thus a projection operator with \li{respect to} the original norm $\| \cdot \|_{2}$. }
Then, it is non-expansive in $\| \cdot \|_{2}$, that is, for any $\tilde{x}$, $\tilde{z}$, 
\begin{align}\label{eq:non-expansiveness of the transformed projection}
    \left\| \mathbf{T} \Pi_{\mathcal{D}}^{\mathbf{P}} ( \mathbf{T}^{-1} \tilde{x})  -  \mathbf{T} \Pi_{\mathcal{D}}^{\mathbf{P}} ( \mathbf{T}^{-1} \tilde{z}) \right\|^2_2 \leq \| \tilde{x} - \tilde{z} \|_2^2.
\end{align}
Therefore, the composition of 
the 
\li{linearly convergent algorithmic operator}
and the non-expansive \li{projection} under the same norm yields the same linear convergence rate. Specifically,
\begin{align*}
    & \left\| \tilde{x}_{k+1} - \tilde{x}_{\li{*}} \right\|_2^2\\
    = & \left\| \mathbf{T} \Pi_{\mathcal{D}}^{\mathbf{P}} ( \mathbf{T}^{-1} \tilde{x}_{k+ \frac{1}{2}}) - \mathbf{T} \Pi_{\mathcal{D}}^{\mathbf{P}} ( \mathbf{T}^{-1} \tilde{x}_{\li{*, \frac{1}{2} }} )  \right\|_2^2\\
    = & \left\| \Pi_{\mathbf{T} \mathcal{D}} ( \tilde{x}_{k+ \frac{1}{2}} )  - \Pi_{\mathbf{T} \mathcal{D}} ( \tilde{x}_{\li{*, \frac{1}{2}}} )\right\|_2^2
    \\
    \leq & \| \tilde{x}_{k + \frac{1}{2}} - \tilde{x}_{\li{*, \frac{1}{2}}} \|_2^2  \leq \rho^2 \| \tilde{x}_{k} - \tilde{x}_{\li{*}} \|_2^2 \leq \ldots\\
    \leq & \rho^{2 (k+1)} \| \tilde{x}_{0} - \tilde{x}_{\li{*}} \|_2^2.
\end{align*}
\li{where $\tilde{x}_{*,\frac{1}{2}} := \mathbf{T} ({x}_{*, \frac{1}{2} }) = \mathbf{T} \left( \hat{A} {x}_* + \hat{B} \nabla f ( y_*) \right)$ and the second inequality follows because the considered IQCs and \eqref{eq: LMI from IQC} hold for any reference point $(y_*, u_*)$.}
As a result, we also have the same convergence rate of the state in the original coordinate but multiplied by a constant, i.e., \mml{it is R-\li{linearly} convergent} with $\| {x}_{k} -  {x}_{\li{*}} \|_2 \leq c \rho^{k} \| {x}_{0} -  {x}_{\li{*}} \|_2$, for some $c > 0$.
\end{proof}
\mll{The convergence analysis relies on Lyapunov theory, the non-expansiveness of projection, and
\li{its composition with the linearly convergent algorithmic operator.}
The quadratic Lyapunov function guarantees R-linear convergence of the unconstrained dynamics, while the weighted projection in Algorithm~\ref{alg:2} is non-expansive only in the weighted norm. Under a linear transformation associated with the Lyapunov matrix, the unconstrained dynamics exhibit Q-linear convergence, and the projection becomes non-expansive in the Euclidean norm. These properties together
establish the same rate bound as that of the unconstrained counterpart.}

\section{Examples}\label{Examples}
\subsection{Projected gradient descent algorithm}
It is well-known that the gradient descent algorithm
\begin{align*}
    x_{k+1} = x_{k} - \li{\alpha} \nabla f(x_{k})
\end{align*}
has the tight convergence rate \mml{bound} $\rho= \frac{L - m}{L + m}$ at $\li{\alpha} = \frac{2}{L + m}$.
The system has $\mathbf{P} = I$ as its Lyapunov function matrix, and thus $\mathbf{P} = \mathbf{T} = I$ in \li{Fig.~\ref{fig: projected algorithm 2}} and \li{Fig.~\ref{fig: augmented system transformed}}, respectively. It follows 
that the projected gradient descent algorithm
\begin{align*}
    x_{k+1} = \Pi_{\Omega} \left( x_{k} - \li{\alpha} \nabla f(x_{k}) \right)
\end{align*}
converges linearly to the optimal solution of \eqref{eq: constrained convex optimization problem} with the same rate $\rho$ using the same stepsize $\li{\alpha}$.

\mll{
\subsection{Projected triple momentum \li{method}}\label{subsection: Projected triple momentum algorithm}
Consider the constrained problem
\begin{align*}
    \min_{y \in \Omega } \left( f(y) = \frac{1}{2} y^{\top} F y + b^{\top} y \right), ~
F =
\begin{bmatrix}
100 & -1\\ -1 & 1
\end{bmatrix},~
b =
\begin{bmatrix}
1 \\ 10
\end{bmatrix}.
\end{align*}
When $\Omega = \mathbb{R}^{2}$, the optimal solution and optimal value are given by
\begin{align*}
y^{\textup{opt}} = - F^{-1} b =
\begin{bmatrix}
-0.1111 \\ -10.1111
\end{bmatrix},
 \quad f^{\textup{opt}} = -50.6111.
\end{align*}
Let us now consider the ellipse constraint 
\[
\Omega = \left\{ y \in \mathbb{R}^{2}: \left(y^{(1)} \right)^2 + 2 \left( y^{(2)}\right)^2  \leq 5 \right\}.
\]
The optimal solution and optimal value are 
\begin{align*}
y^{\textup{opt}}_{\Omega} = \begin{bmatrix}
-0.0251 \\ -1.5810
\end{bmatrix},
\quad f^{\textup{opt}}_{\Omega} = -14.5938
\end{align*}
obtained by solving the constrained problem \li{analytically via the KKT condition\cite{ruszczynski2011nonlinear}}.}
\mll{
The triple momentum method \li{(TMM)} is given by
\begin{align*}
    \xi_{k+1} = & (1 + \beta) \xi_{k} - \beta \xi_{k-1} - \alpha \nabla f (y_k)\\
    y_{k} = & (1 + \gamma) \xi_{k} - \gamma \xi_{k-1}
\end{align*}
which can be written as
\begin{align*}
    \begin{bmatrix}
    \begin{array}{c}
        \xi^{(1)}_{k+1} \\ \xi^{(2)}_{k+1} \\ \hline y_{k}
        \end{array}
    \end{bmatrix}
    = 
    \begin{bmatrix}
    \begin{array}{cc|c}
       \left( 1 + \beta \right) I_{d} & -\beta I_{d} & - \alpha I_{d}\\
       I_{d} & 0_{d} & 0_{d} \\
       \hline
       (1 + \gamma) I_{d} & -\gamma I_{d} & 0_{d}
       \end{array}
    \end{bmatrix}
    \begin{bmatrix}
    \begin{array}{c}
        \xi^{(1)}_{k} \\ \xi^{(2)}_{k} \\ \hline u_{k}
    \end{array}
    \end{bmatrix}
\end{align*}
with $u_{k} = \nabla f(y_{k})$. 
We adopt the parameters $(\alpha, \beta, \gamma) = \left(\frac{1 + \rho}{L}, \frac{\rho^2}{2 - \rho}, \frac{\rho^2}{(1+ \rho) (2 - \rho)} \right)$ from \cite{van2017fastest}, where $m = 0.9899$, $L = 100.0101$, and the convergence rate is $\rho = 1 - \sqrt{m/L} = 0.9005$ for this example.}
\mll{The state-space system can be transformed further to include output as a part of the state vector that directly receives the input exclusively, that is,
\begin{align*}
    &
    \begin{bmatrix}
    \begin{array}{c}
        y_{k+1} \\ \xi^{(2)}_{k+1} \\ \hline y_{k}
        \end{array}
    \end{bmatrix} = 
    \nonumber \\
    &
    \hspace{-1mm}
    \begin{bmatrix}
    \begin{array}{cc|c}
       \frac{ \left( (\beta + 1)(\gamma + 1) - \gamma\right) I_{d} }{\gamma + 1}  & \frac{ \left( \gamma - \beta - \beta \gamma \right) I_{d} }{\gamma + 1} & -\alpha\left(\gamma + 1\right) I_{d} \\
       \frac{I_{d}}{\gamma + 1}  & \frac{\gamma  I_{d}}{\gamma + 1} & 0_{d} \\
       \hline
        I_{d} & 0_{d} & 0_{d}
       \end{array}
    \end{bmatrix}
    \hspace{-1.5mm}
    \begin{bmatrix}
    \begin{array}{c}
        y_{k} \\ \xi^{(2)}_{k} \\ \hline u_{k}
    \end{array}
    \end{bmatrix} \hspace{-1.5mm}.
\end{align*}
Then, \li{a direct projected step as in \eqref{alg Euclidean proj}} is constructed as
\begin{align}
    &
    \begin{bmatrix}
    \begin{array}{c}
        y_{k+\frac{1}{2}} \\ \xi^{(2)}_{k+1} \\ \hline y_{k}
        \end{array}
    \end{bmatrix} = \nonumber
    \\
    &
    \hspace{-1.5mm}
    \begin{bmatrix}
    \begin{array}{cc|c}
       \frac{ \left( (\beta + 1)(\gamma + 1) - \gamma\right) I_{d} }{\gamma + 1}  & \frac{ \left( \gamma - \beta - \beta \gamma \right) I_{d} }{\gamma + 1} & -\alpha\left(\gamma + 1\right) I_{d} \\
       \frac{I_{d}}{\gamma + 1}  & \frac{\gamma  I_{d}}{\gamma + 1} & 0_{d} \\
       \hline
        I_{d} & 0_{d} & 0_{d}
       \end{array}
    \end{bmatrix}
    \hspace{-1.5mm}
    \begin{bmatrix}
    \begin{array}{c}
        y_{k} \\ \xi^{(2)}_{k} \\ \hline u_{k}
    \end{array}
    \end{bmatrix} \label{eq: Projected TM Euclidean}
\end{align}
with $u_{k} = \nabla f(y_{k})$ and $y_{k+1} = \Pi_{\Omega} \left(y_{k + \frac{1}{2}} \right)$, which can be written compactly to
\begin{align}\label{eq: algorithm 1 TM}
    y_{k+1} = & \Pi_{\Omega} \left( \frac{ \left( (\beta + 1)(\gamma + 1) - \gamma\right) }{\gamma + 1}  y_{k} + \frac{ \left( \gamma - \beta - \beta \gamma \right) }{\gamma + 1} \xi^{(2)}_{k} \right. \nonumber\\
    & \quad \quad \left. \vphantom{\frac{ \left( (\beta + 1)(\gamma + 1) - \gamma\right) }{\gamma + 1}} - \alpha \left( \gamma + 1\right) \nabla f \left( y_{k} \right) \right), \nonumber\\
    \xi^{(2)}_{k+1} = & \frac{1 }{\gamma + 1} y_{k} + \frac{\gamma }{\gamma + 1} \xi^{(2)}_{k}.
\end{align}
}


\mll{
It can be verified through PESTO \cite{taylor2017performance}, a toolbox derived from the performance estimation problems \cite{taylor2017exact}, that algorithm \eqref{eq: algorithm 1 TM} for this example is convergent. However, the same convergence rate bound is not guaranteed.}

\mll{Next, to construct Algorithm~\ref{alg:2}, we adopt three IQCs \li{in \cite{lessard2016analysis}}, namely, the sector-bound IQC, off-by-one IQC and the weighted off-by-one IQC with $\bar{\rho} = \rho$. \li{The use of these three IQCs is justifiable by the fact that they satisfy \eqref{eq: time domain discrete-time IQC} as well as \eqref{eq: truncated hard IQC}, and the reference point $(h_*, y_*, u_*)$ is independent of the fixed point of the unconstrained algorithm.}
Then, the auxiliary linear systems are given by
\begin{align*}
\Psi_{\textup{SB}}:
	 & \begin{bmatrix}
		 D_{\Psi_{\textup{SB}}}^{y} & D_{\Psi_{\textup{SB}}}^{u}
	\end{bmatrix}
	= \begin{bmatrix}
		L I_d & -I_d\\
		-m I_d & I_d
	\end{bmatrix},\\
\Psi_{\textup{Off}}:
	 & \begin{bmatrix}
		\begin{array}{c|cc}
		A_{\Psi_{\textup{Off}}} & B_{\Psi_{\textup{Off}}}^{y} & B_{\Psi_{\textup{Off}}}^{u}\\
		\hline
		C_{\Psi_{\textup{Off}}} & D_{\Psi_{\textup{Off}}}^{y} & D_{\Psi_{\textup{Off}}}^{u}
		\end{array}
	\end{bmatrix}
	= \begin{bmatrix}
		\begin{array}{c|cc}
		0_d & -L I_d & I_d\\
		\hline
		I_d & L I_d & -I_d\\
		0_d & -m I_d & I_d
		\end{array}
	\end{bmatrix}\\
\Psi_{\textup{WO}}:
	 & \begin{bmatrix}
		\begin{array}{c|cc}
		A_{\Psi_{\textup{WO}}} & B_{\Psi_{\textup{WO}}}^{y} & B_{\Psi_{\textup{WO}}}^{u}\\
		\hline
		C_{\Psi_{\textup{WO}}} & D_{\Psi_{\textup{WO}}}^{y} & D_{\Psi_{\textup{WO}}}^{u}
		\end{array}
	\end{bmatrix}
	= \begin{bmatrix}
		\begin{array}{c|cc}
		0_d & - L I_d & I_d\\
		\hline
		\bar{\rho}^2 I_d & L I_d & -I_d\\
		0_d  &-m I_d & I_d
		\end{array}
	\end{bmatrix}
\end{align*}
respectively, with 
\begin{align*}
    M_{\textup{SB}} = M_{\textup{Off}} = M_{\textup{WO}} = \begin{bmatrix}
        0 & 1\\ 1 & 0
    \end{bmatrix} \otimes I_{d}.
\end{align*}
We denote 
\begin{align*}
    & 
    A_{\Psi} = \operatorname{blkdiag}\left( A_{\Psi_{\textup{Off}}}, A_{\Psi_{\textup{WO}}} \right), 
    B_{\Psi}^{y} = \begin{bmatrix}
        B_{\Psi_{\textup{Off}}^{y}}\\ B_{\Psi_{\textup{WO}}^{y}}
    \end{bmatrix},\\
    & 
    B_{\Psi}^{u} = \begin{bmatrix}
        B_{\Psi_{\textup{Off}}^{u}}\\ B_{\Psi_{\textup{WO}}^{u}}
    \end{bmatrix},
    C_{\Psi} = \operatorname{blkdiag}\left( C_{\Psi_{\textup{Off}}}, C_{\Psi_{\textup{WO}}} \right), \\
    & 
    D_{\Psi}^{y} = \begin{bmatrix}
        D_{\Psi_{\textup{SB}}^{y}}\\ D_{\Psi_{\textup{Off}}^{y}}\\ D_{\Psi_{\textup{WO}}^{y}}
    \end{bmatrix},
    D_{\Psi}^{u} = \begin{bmatrix}
        D_{\Psi_{\textup{SB}}^{u}}\\ D_{\Psi_{\textup{Off}}^{u}}\\ D_{\Psi_{\textup{WO}}^{u}}
    \end{bmatrix},\\
    &
    M = \operatorname{blkdiag} \left( \lambda_1 M_{\textup{SB}}, \lambda_2 M_{\textup{Off}}, \lambda_3 M_{\textup{WO}} \right)
\end{align*}
where $\operatorname{blkdiag}(\cdot)$ represents the block diagonal operator and $\lambda_i \geq 0$, $i = 1, 2, 3$, are decision variables.
Then, we stack all the states together by \eqref{eq: augmented system with IQC} and \eqref{eq: IQC augmented system matrices}.
By solving \eqref{eq: LMI from IQC}, we obtain
\begin{align*}
    {P} = &
    \li{\begin{bmatrix}
    1046.220 & 1209.476 &   11.315 & 10.925\\
    1209.476 & 1788.767 &  13.908   & 15.995\\
    11.315  & 13.908 & 233.505 & -233.255\\
    10.925  & 15.995 &-233.255 & 233.523\\
    \end{bmatrix}
    }\\
    := &
    \begin{bmatrix} {P}_{11} & {P}_{12} \\ {P}_{12}^{\top} & {P}_{22}
    \end{bmatrix}, \quad {P}_{11} \in \mathbb{R}.
\end{align*}
Then, Algorithm~\ref{alg:2} is constructed as
\begin{equation}\label{eq: Projected TM P norm}
\begin{aligned}
    y_{k + \frac{1}{2}} = & \frac{ \left( (\beta + 1)(\gamma + 1) - \gamma\right) }{\gamma + 1}  y_{k} + \frac{ \left( \gamma - \beta - \beta \gamma \right) }{\gamma + 1} \xi^{(2)}_{k} \\
    & - \alpha \left( \gamma + 1\right) \nabla f \left( y_{k} \right), \\
    y_{k+1} = & \Pi_{\Omega} \left( y_{k + \frac{1}{2}} \right),\\
    \xi^{(2)}_{k+1} = & \frac{1 }{\gamma + 1} y_{k} + \frac{\gamma }{\gamma + 1} \xi^{(2)}_{k} - \mathcal{\chi} \left( y_{k+1} - y^{(1)}_{k+\frac{1}{2}} \right)
\end{aligned}
\end{equation}
where $\li{\chi = -1.209}$ is the first element of $ {P}_{22}^{-1}  {P}_{12}^{\top}$. The parameters $(\alpha, \beta, \gamma)$ remain unchanged from the unconstrained case.
}

\mll{
The trajectory of $\left\| y_{k} - y^{\textup{opt}}\right\|_2$ for the unconstrained TMM is shown in \li{Fig.~\ref{fig: numerical examples}\subref{fig: unconstrained PM}}.
The trajectories of $\left\| y_{k} - y^{\textup{opt}}_{\Omega} \right\|_2$ for algorithm \eqref{eq: Projected TM Euclidean} and \eqref{eq: Projected TM P norm} are shown in \li{Fig.~\ref{fig: numerical examples}\subref{fig: projected TM algorithm comparison 1}}. We also include the trajectories of the projected gradient algorithm for comparison and to illustrate the precision limits of the solver.
\begin{figure}[htbp]
     \centering
\begin{subfigure}[t]{1\linewidth}
    \centering
    \includegraphics[width=1\linewidth]{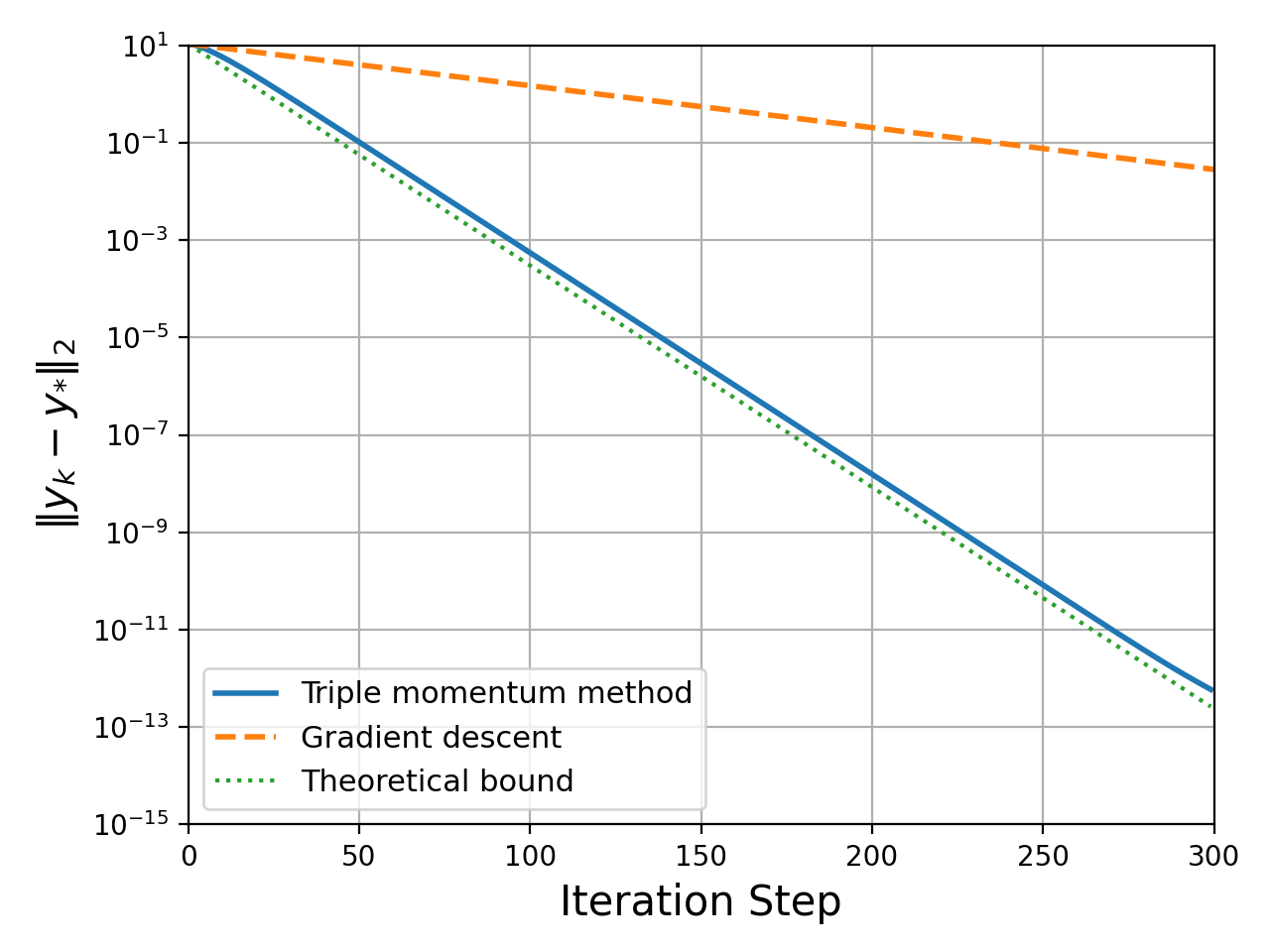}
    \caption{\li{Distance $\| y_{k} - y_{*} \|$ versus the iteration step for gradient descent and the unconstrained TMM.}}
    \label{fig: unconstrained PM}
\end{subfigure}
\begin{subfigure}[t]{1\linewidth}
    \centering
    \includegraphics[width=1\linewidth]{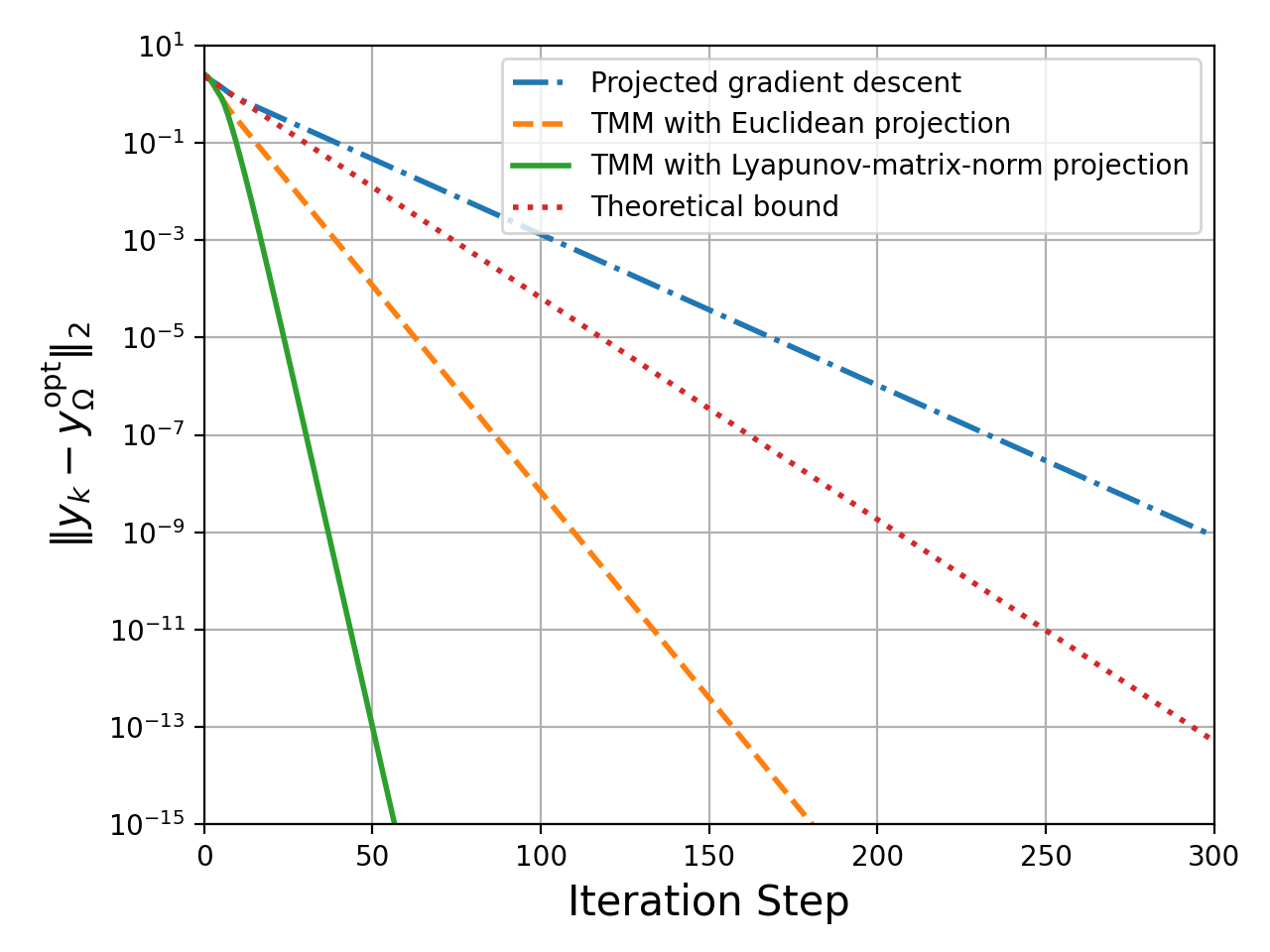}
    \caption{Distance $\| y_{k} - y^{\textup{opt}}_{\Omega} \|$ versus the iteration step for projected gradient descent method, \li{the TMM with direct Euclidean projection~\eqref{eq: Projected TM Euclidean}, and the TMM with projection in the Lyapunov-matrix norm~\eqref{eq: Projected TM P norm}.}}
    \label{fig: projected TM algorithm comparison 1}
\end{subfigure}
\caption{Convergence error for the triple momentum \li{method} and its projected variants, together with the theoretical rate bound $\rho = 1 - \sqrt{m/L}$.}
        \label{fig: numerical examples}
\end{figure}
It can be observed that the algorithms \eqref{eq: Projected TM Euclidean} and \eqref{eq: Projected TM P norm} exhibit faster convergence than the unconstrained dynamics and the corresponding theoretical rate bound. This is expected, as the projection step may pull each iterate closer to the optimal point and accelerate the convergence, even though analysis usually treats the projection as a non-expansive operator. Furthermore, algorithm~\eqref{eq: Projected TM P norm} converges faster \li{than algorithm~\eqref{eq: Projected TM Euclidean}}, likely due to the acceleration introduced by the momentum term.
The source code used for this numerical example is available online~\cite{li2026projectedTMMcode}.
}

\section{Conclusion}\label{Conclusion}
We have proposed a systematic procedure to construct first-order projected algorithms on top of the unconstrained optimization algorithms in the Lur'e form.
We have proved that the proposed projected algorithms have the same convergence rate bounds obtained through a quadratic Lyapunov function associated with IQCs for the gradient dynamics of the unconstrained algorithms. This finding significantly reduces the effort required to analyze first-order algorithms in that one only needs to consider unconstrained problems and their dynamics. 
\icl{Future work includes investigating} 
whether a similar analysis could be carried out for other projected algorithms, such as the projected Newton method and mirror descent algorithms, \ml{where the system matrices involved in the algorithm may not admit a Kronecker structure}. Meanwhile, designing a similar procedure for continuous-time algorithms could be an interesting extension.


\section*{Appendix}

\subsection{Proof of Lemma~\ref{lem: property of the system matrix A}}\label{appendix proof of Lemma property of the system matrix A}
\li{Let $d = 1$ without loss of generality. We first show that $I - \tilde{A}_{22}$ is nonsingular.
    By \cite{scherer2021convex}, the closed-loop system matrix $\tilde{A} + m \tilde{B} \tilde{C}$ is Schur stable, then, $\det \left( I - \left( \tilde{A} + m \tilde{B}\tilde{C} \right) \right) > 0$.  From \eqref{eq: state-space form with state output for unconstrained optimization}, we have $\tilde{B} \tilde{C} = \begin{bmatrix}
    c_1 & 0 \\ 0 & 0
\end{bmatrix}$, then
\begin{align}\label{eq: determinant of IAmBC}
\begin{aligned}
    & \det \left( I - \left( \tilde{A} + m \tilde{B}\tilde{C} \right) \right)\\
    = & \det \left( \begin{bmatrix}
        1 - \left( \tilde{A}_{11} + m c_1 \right) & -\tilde{A}_{12}\\ - \tilde{A}_{21} & I - \tilde{A}_{22}
    \end{bmatrix} \right)\\
    = & \left(1 - \tilde{A}_{11} - m c_1 \right) \det \left( I - \tilde{A}_{22} \right) + \phi > 0.
\end{aligned}
\end{align}
where $\phi$ denotes the rest of the terms independent of $m$. 
It is known that $\tilde{A}$ has one eigenvalue at one (when $d = 1$), and the rest are strictly within the unit disc \cite{lessard2016analysis}, which implies 
\begin{align}\label{eq: determinant of I-A}
\det \left( I -  \tilde{A} \right) = \left(1 - \tilde{A}_{11} \right) \det \left( I - \tilde{A}_{22} \right) + \phi = 0.
\end{align}
Suppose that $I - \tilde{A}_{22}$ is singular, that is, $\det \left( I -  \tilde{A}_{22} \right) = 0$. Then $\det \left( I -  \tilde{A} \right) = \det \left( I - \left( \tilde{A} + m \tilde{B}\tilde{C} \right) \right) = \phi = 0$, regardless of $m$, which is a contradiction. Thus, $I - \tilde{A}_{22}$ is nonsingular.
Eliminating $\phi$ from \eqref{eq: determinant of I-A}, \eqref{eq: determinant of IAmBC}, we obtain
\begin{align*}
    \det \left( I - \left( \tilde{A} + m \tilde{B}\tilde{C} \right) \right) = - m c_1 \det \left( I - \tilde{A}_{22} \right) > 0.
\end{align*}
By Assumption~\ref{assumption on CB}, $c_1 <0$. Therefore, $\det \left( I - \tilde{A}_{22} \right) > 0$.\\
Finally, since \(I-\tilde A_{22}\) is nonsingular, the Schur complement formula applied to \(I-\tilde A\) gives
\begin{align*}
& \det(I-\tilde A)\\
= &
\det(I-\tilde A_{22})
\det \left(I-\tilde A_{11}-\tilde A_{12}(I-\tilde A_{22})^{-1}\tilde A_{21}\right) = 0
\end{align*}
Because \(\det(I-\tilde A_{22})\neq 0\), it follows that
\[
I-\tilde A_{11}-\tilde A_{12}(I-\tilde A_{22})^{-1}\tilde A_{21}=0,
\]
which yields \eqref{eq: property of the system matrix A}.}

\subsection{Proof of Lemma~\ref{lem: Schur of transformed A_22}}\label{appendix proof of Lemma Schur of transformed A_22}
\li{Let us consider a coordinate transformation $\bar{A} = T A T^{-1}$, where
\begin{align*}
    T = \begin{bmatrix}
        1 & 0\\ P_{22}^{-1} P_{12}^\top & I
    \end{bmatrix}, \quad
    T^{-1} = \begin{bmatrix}
        1 & 0\\ - P_{22}^{-1} P_{12}^\top & I
    \end{bmatrix}
\end{align*}
such that $P$ is diagonalized in the new coordinate, that is,
\begin{align*}
    P = T^\top \underbrace{\begin{bmatrix}
        P_{11} - P_{12} P_{22}^{-1} P_{12}^\top & 0\\ 0 & P_{22} 
    \end{bmatrix}}_{\bar{P}} T.
\end{align*}
Then, we have
\begin{align}\label{eq: diagonalized Lyapunov inequality}
    \bar{A}^\top \bar{P} \bar{A} \preceq \rho^2 \bar{P}.
\end{align}
By direct calculation,
\[
\bar A_{22}=A_{22}+P_{22}^{-1}P_{12}^\top A_{12}.
\]
Now, the lower-right block of \eqref{eq: diagonalized Lyapunov inequality} gives
\[
\bar A_{12}^\top \bar P_{11}\bar A_{12}
+
\bar A_{22}^\top \bar P_{22}\bar A_{22}
\preceq
\rho^2 \bar P_{22}.
\]
Since \(\bar P_{11}\succ0\) by the positive definiteness of $\bar{P}$, we have
\[
\bar A_{22}^\top \bar P_{22}\bar A_{22} \preceq \rho^2 \bar P_{22}.
\]
As \(\bar P_{22}=P_{22}\succ0\) and \(\rho<1\), it follows that \(\bar A_{22}\) is Schur stable.
It can be shown similarly that $\bar{A}_{11} = A_{11} - A_{12} P_{22}^{-1} P_{12}^\top$ is also Schur stable. This completes the proof.
}

\section*{References}
\bibliographystyle{IEEEtran}
\bibliography{References}

\begin{IEEEbiography}[{\includegraphics[width=1in,height=1.25in,clip,keepaspectratio]{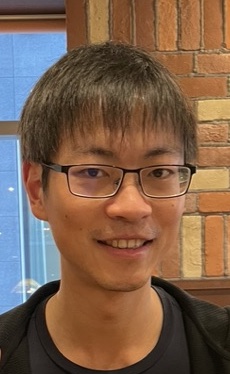}}]{Mengmou Li} (Member, IEEE) 
is currently a Tenure-Track Associate Professor at the Graduate School of Advanced Science and Engineering, Hiroshima University, Japan. He received the B.S. degree in Physics from Zhejiang University, China, in 2016, and the Ph.D. degree in Electrical and Electronic Engineering from The University of Hong Kong in 2020. He held postdoctoral positions at The Hong Kong University of Science and Technology, Hong Kong, the Control Group at the University of Cambridge, UK, and Tokyo Institute of Technology, Japan. He was a Specially Appointed Assistant Professor at Tokyo Institute of Technology. His research interests include cyber-physical systems, optimization, and robust control.
\end{IEEEbiography}

\begin{IEEEbiography}[{\includegraphics[width=1in,height=1.25in,clip,keepaspectratio]{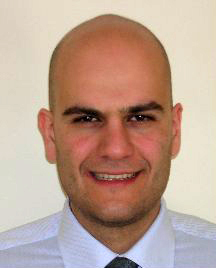}}]{Ioannis Lestas}(Member, IEEE)
is a Professor of Control Engineering at the Department of Engineering, University of Cambridge. He received the B.A. (Starred First) and M.Eng. (Distinction) degrees in Electrical Engineering and Information Sciences and the Ph.D. in control engineering from the University of Cambridge (Trinity College). His doctoral work was performed as a Gates Scholar. He has been a Junior Research Fellow of Clare College, University of Cambridge and he was awarded a five year Royal Academy of Engineering research fellowship. He is also the recipient of a five year ERC starting grant, and an ERC proof of concept grant. He is currently serving as Associate Editor for the IEEE Transactions on Automatic Control, the IEEE Transactions on Smart Grid, and the IEEE Transactions on Control of Network Systems. His research interests include analysis and control of large-scale networks with applications in power systems and smart grids.
\end{IEEEbiography}

\begin{IEEEbiography}[{\includegraphics[width=1in,height=1.25in,clip,keepaspectratio]{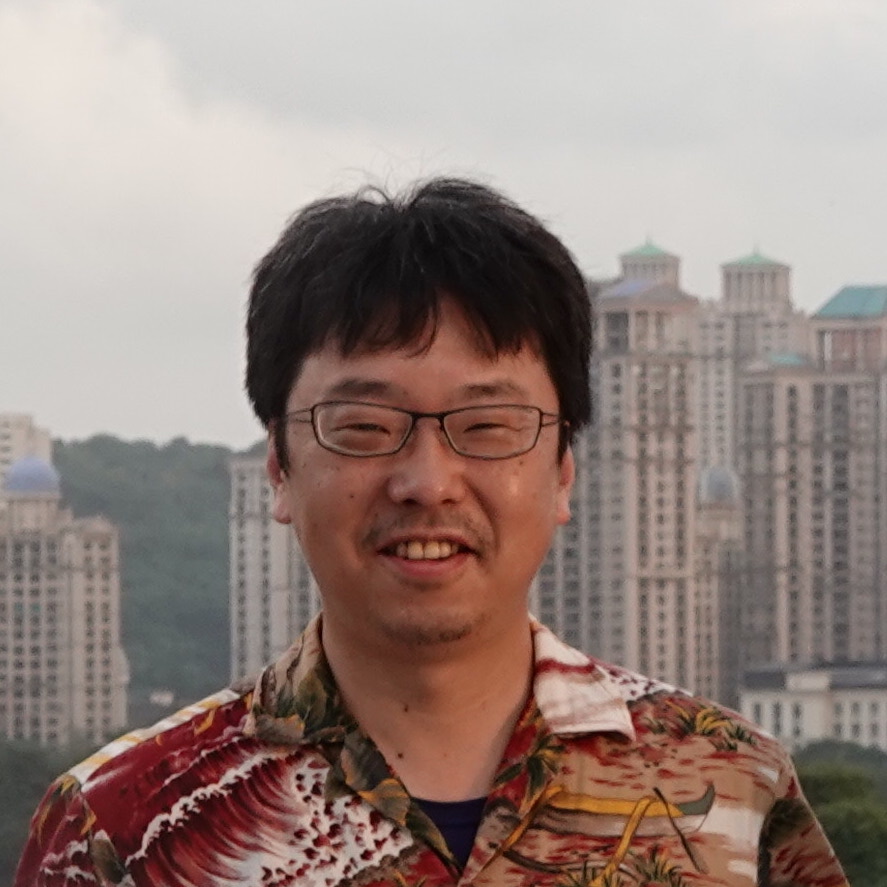}}]{Masaaki Nagahara}(Senior Member, IEEE)
received a bachelor's degree in engineering from Kobe University in 1998 and a master's degree and a Doctoral degree in informatics from Kyoto University in 2000 and 2003, respectively. He is currently a Full Professor at the Graduate School of Advanced Science and Engineering, Hiroshima University. He has been a Visiting Professor at Indian Institute of Technology Bombay since 2017. His research interests include control theory, machine learning, and sparse modeling. He received remarkable international awards: Transition to Practice Award in 2012 and George S. Axelby Outstanding Paper Award in 2018 from the IEEE Control Systems Society. Also, he received many awards from Japanese research societies, such as SICE Young Authors Award in 1999, SICE Best Paper Award in 2012, SICE Best Book Authors Awards in 2016 and 2021, SICE Control Division Research Award (Kimura Award) in 2020, and the Best Tutorial Paper Award from the IEICE Communications Society in 2014. He is a senior member of IEEE, and a member of IEICE, SICE, ISCIE, and RSJ.
\end{IEEEbiography}

\end{document}